\documentclass{amsart}

\usepackage{amsthm,amsmath,amssymb,amsfonts,dsfont}
\usepackage[colorlinks, linkcolor=blue, citecolor=blue]{hyperref}
\usepackage[indentafter]{titlesec}
\titleformat{name=\section}{}{\thetitle.}{0.8em}{\centering\scshape}
\usepackage{mathrsfs}
\usepackage{bm}
\usepackage{titlesec}
\usepackage{xurl}
\numberwithin{equation}{section}

\newtheorem{definition}{\textbf{Definition}}[section]
\newtheorem{theorem}[definition]{\textbf{Theorem}}
\newtheorem{corollary}[definition]{\textbf{Corollary}}
\newtheorem{lemma}[definition]{\textbf{Lemma}}
\newtheorem{proposition}[definition]{\textbf{Proposition}}
\newtheorem{example}[definition]{\textbf{Example}}

\newtheorem{letterthm}{\textbf{Theorem}}

\usepackage{thmtools}
\declaretheoremstyle[bodyfont=\normalfont]{normalbody}
\declaretheorem[numberlike=definition,style=normalbody,name=Remark]{remark}
\swapnumbers
\declaretheorem[numberwithin=section,style=normalbody,name=Hyperstates and u.c.p. maps]{hyperstate}
\declaretheorem[numberlike=hyperstate,style=normalbody,name=Noncommutative Poisson boundaries]{NCPB}
\declaretheorem[numberlike=hyperstate,style=normalbody,name=Entropy]{entropy}
\declaretheorem[numberlike=hyperstate,style=normalbody, refname={2.4}, name=Ultraproducts of von Neumann algebras]{Ultraproduct}

\def\d{\mathrm{d}}

\def\CA{C^\alpha(S^1)}

\def\supp{\mathrm{supp \, }}
\def\suppess{\mathrm{supp}_{\mathrm{ess}}\,}
\def\prob{\mathrm{Prob}}

\def\CB{\mathcal{B}}
\def\CP{\mathcal{P}}
\def\CA{\mathcal{A}}
\def\har{\mathrm{Har}}
\def\id {\mathrm{id}}

\title{Noncommutative Poisson boundaries, ultraproducts and entropy}
\author{Shuoxing Zhou}
\address{Universit\'e Paris-Saclay \\  Laboratoire de Math\'ematiques d'Orsay\\ 91405 Orsay\\ FRANCE}
\email{shuoxing.zhou@universite-paris-saclay.fr}

\begin{document}

\maketitle
\begin{abstract}
We construct the noncommutative Poisson boundaries of tracial von Neumann algebras through the ultraproducts of von Neumann algebras. As an application of this result, we complete the proof of Kaimanovich-Vershik's fundamental theorems regarding noncommutative entropy. We also prove the Amenability-Trivial Boundary equivalence and Choquet-Deny-Type I equivalence for tracial von Neumann algebras.
\end{abstract}

\section{Introduction}

Recently, Das-Peterson \cite{DP20} extended the notion of Poisson boundary to the noncommutative setting: Let $(M,\tau)$ be a tracial von Neumann algebra with separable predual. As a generalization of admissible measures on discrete groups, a normal regular strongly generating hyperstate on $B(L^2(M))$ is a normal state $\varphi\in B(L^2(M))_*$ with the standard form $\varphi(T)=\sum_{n=1}^{\infty}\langle T  \hat{z}_n^*,\hat{z}_n^*\rangle$, where $\{z_n\}\subset M$ satisfying $\sum_{n=1}^{\infty}z_n^*z_n=\sum_{n=1}^{\infty} z_nz_n^*=1$ and $\{z_n\}$ generates $M$ as a weakly closed unital algebra. The normal u.c.p. $M$-bimodular map associated to $\varphi$ is defined as $\CP_\varphi:T\in B(L^2(M))\mapsto\sum_{n=1}^{\infty} (Jz_n^*J)T(Jz_nJ)\in B(L^2(M))$. As a generalization of classical Poisson boundary, the noncommutative Poisson boundary of $\varphi$ is the unique von Neumann algebra $\CB_\varphi$ that is isomorphic to $\har(\CP_\varphi)=\{ T\in B(L^2(M))\mid \mathcal{P}_\varphi(T)=T\}$ as operator system. The isomorphism $\CP:\CB_\varphi\to \har(\CP_\varphi)$ is called the Poisson transform. Note that since $M\subset\mathrm{Har}(\mathcal{P}_\varphi)$, $M$ can always be embedded into $\CB_\varphi$ as a subalgebra. And $\zeta:=\varphi\circ\CP$ is the canonical hyperstate on $\CB_\varphi$. 

The following example \cite[Theorm 4.1]{Iz04} shows that the noncommutative 
Poisson boundary is a generalization of the classical one: Let $(M,\tau)=(L(\Gamma),\tau)$ be the group von Neumann algebra of a countable discrete group $\Gamma$, and $\varphi (T)=\sum_{\gamma\in\Gamma}\mu(\gamma)\langle T\delta_{\gamma^{-1}},\delta_{\gamma^{-1}}\rangle$ be the hyperstate associated to $\mu\in\prob(\Gamma)$. Then the $\varphi$-Poisson boundary $\CB_\varphi$ is exactly the group measure space construction $L(\Gamma\curvearrowright B)$, where $(B,\nu_B)$ is the $(\Gamma,\mu)$-Poisson boundary. 

Das-Peterson \cite{DP20} already proved several analogues of the classical Poisson boundary theory, for examples, the injectivity, double ergodicity and rigidity of $\CB_\varphi$. The main goal of this paper is to further explore the analogues between classical and noncommutative Poisson boundaries.

Das-Peterson \cite{DP20} also initiated entropy theory of noncommutative
boundaries, which we refer to \ref{NC entropy} for a brief introduction of the definition. Moreover, Das-Peterson \cite{DP20} proved the analogues of part of Kaimanovich-Vershik's fundamental theorems regarding entropy \cite{KV82}, while the rest are left open. In Section 3, inspired by \cite{SS22}, we develop a new way to construct $\CB_\varphi$ with the ultraproducts of von Neumann algebras, which helps completing the proof of Kaimanovich-Vershik's fundamental theorems:
 
\begin{letterthm}
    Let $\varphi\in \mathcal{S}_\tau(B(L^2(M,\tau))$ be a normal regular strongly generating hyperstate such that $H(\varphi)<+\infty$. Take an increasing sequence $\{a_n\}\subset\mathbb{R}$ satisfying $\frac{1}{2}\leq a_n< 1 $ and $\lim_n\limits a_n=1$, and an ultrafilter $\omega\in \beta \mathbb{N}\setminus \mathbb{N}$. Define a sequence of von Neumann algebras with faithful normal hyperstates by $$(\CA_n,\varphi_n)=(B(L^2(M)),(1-a_n)\sum_{k=0}^{\infty}a_n^k \varphi^{*k}).$$
    Then the $\varphi$-Poisson boundary with canonical hyperstate $(\CB_\varphi,\zeta)$ can be embedded into the Ocneanu ultraproduct $(\CA^\omega, \varphi^\omega):=(\CA_n,\varphi_n)^\omega$ as a von Neumann subalgebra  with normal conditional expectation and the following facts hold:
\begin{itemize}
    \item [(i)] $h_\varphi(\CB_\varphi,\zeta)=h_\varphi(\CA^\omega, \varphi^\omega)=\lim_\omega\limits h_\varphi(\CA_n, \varphi_n)=h(\varphi)$.
    \item [(ii)] A $\varphi$-boundary $(\CB_0,\zeta_0)$ (i.e. von Neumann subalgebra of $(\CB_\varphi,\zeta)$ that contains $(M,\tau)$) satisfies $h_\varphi(\CB_0,\zeta_0)= h_\varphi(\CB_\varphi,\zeta)$ if and only if $(\CB_0,\zeta_0)=(\CB_\varphi,\zeta)$.
    \item [(iii)] $(\CB_\varphi,\zeta)$ is trivial (i.e. $(\CB_\varphi,\zeta)=(M,\tau)$) if and only if $h(\varphi)=0$.
\end{itemize}
\end{letterthm}

In \cite{KV82}, Kaimanovich-Vershik provided an important characterization of amenable groups involving classical Poisson boundary: A countable discrete group $\Gamma$ is amenable if and only if there exists an admissible measure $\mu \in \prob(\Gamma)$ such that the Poisson boundary of $(\Gamma,\mu)$ is trivial. In Section 4, we prove an analogue of Kaimanovich-Vershik's amenability-trivial boundary equivalence:

\begin{letterthm}
    The tracial von Neumann algebra $(M,\tau)$ is amenable if and only if there exists a measure $\mu\in\prob(\mathcal{U}(M))$ such that the associated hyperstate $\varphi_\mu(T)=\int_{\mathcal{U}(M)} \langle T u^*\hat{1},u^*\hat{1}\rangle \d \mu(u)$ has trivial Poisson boundary.
\end{letterthm}

The Choquet-Deny property of locally compact groups is a property stronger than amenability: A locally compact group $G$ is Choquet-Deny if for any admissible $\mu\in\prob(G)$, the $\mu$-Poisson boundary is trivial. Since the normal regular strongly generating hyperstate is a generalization of admissible measure, it's natural to define the Choquet-Deny property of tracial von Neumann algebras to be that any normal regular strongly generating hyperstate has trivial Poisson boundary. 

The classical Choquet-Deny theorem \cite{CD60} states that any abelian countable discrete group is Choquet-Deny, of which Das-Peterson \cite{DP20} proved an analogue, telling the same for abelian tracial von Neumann algebras. In Section 5, inspired by the latest progress regarding Choquet-Deny property of groups (\cite{Ja04} and \cite{FHTV19}), we prove a stronger result, providing a more comprehensive understanding of the Choquet-Deny property of tracial von Neumann algebras:

\begin{letterthm}
    Let $(M,\tau)$ be a tracial von Neumann algebra with separable predual. Then $(M,\tau)$ is Choquet-Deny if and only if $M$ is of type $\mathrm{I}$.  
\end{letterthm}

Following \cite[Theorm 4.1]{Iz04}, for a countable discrete group $\Gamma$, the Choquet-Deny property of $L(\Gamma)$ can always induce the Choquet-Deny property of $\Gamma$. Note that when $\Gamma$ is finitely generated, $\Gamma$ is Choquet-Deny if and only if $\Gamma$ is virtually nilpotent (see \cite{Ja04} and \cite{FHTV19}). And $L(\Gamma)$ is of type $\mathrm{I}$ if and only if $\Gamma$ is virtually abelian (see \cite{THOM64}). The Heisenberg group $H_3(\mathbb{Z})$ is exactly such a finitely generated group that is virtually nilpotent but not virtually abelian. Hence for a countable discrete group $\Gamma$, the Choquet-Deny property of $L(\Gamma)$ is not equivalent to, but strictly stronger than the Choquet-Deny property of $\Gamma$. That's because the construction of the noncommutative Poisson boundary is more flexible than the construction of the classical Poisson boundary.

\section{Preliminaries}\label{Preliminaries}

\begin{hyperstate}
Fix a tracial von Neumann algebra $(M, \tau )$. Following \cite{DP20}, for a $\mathrm{C}^*$-algebra $\CA$ such that $M\subset\CA$ and a state $\psi \in \mathcal{S}(\CA)$, we say that $\psi$ is a $\tau$-\textbf{hyperstate} if $\psi|_M =\tau$. We denote by $\mathcal{S}_\tau(\CA)$ the set of $\tau$-hyperstates on $\CA$. Let $e_M\in B(L^2(\CA,\psi))$ be the orthogonal projection onto $L^2(M)$. The u.c.p. map $\CP_\psi:\CA\to B(L^2(M))$ is defined as 
$$\CP_\psi(T)=e_MTe_M, \ T\in \CA.$$
According to \cite[Proposition 2.1]{DP20}, $\psi\mapsto\CP_\psi$ is a bijection between (normal) hyperstates on $\CA$ and (normal) u.c.p. $M$-bimodular maps from $\CA$ to $B(L^2(M))$, whose inverse is $\CP \mapsto \langle\CP(\,\cdot\,)\hat{1},\hat{1}\rangle$. 

For $\psi\in\mathcal{S}_\tau(\CA)$ and $\varphi \in \mathcal{S}_\tau(B(L^2(M)))$, the convolution $\varphi \ast\psi\in \mathcal{S}_\tau(\CA)$ is defined to be the hyperstate associated to the $M$-bimodular u.c.p. map $\mathcal{P}_{\varphi}\circ \mathcal{P}_{\psi}$. And $\psi$ is said to be $\varphi$-stationary if $\varphi \ast\psi=\psi$.
\end{hyperstate}

\begin{NCPB}
Let $\varphi\in\mathcal{S}_\tau(B(L^2(M)))$ be a hyperstate. Following \cite{DP20}, the set of $\mathcal{P}_\varphi$\textbf{-harmonic operators} is defined to be
$$\mathrm{Har}(\mathcal{P}_\varphi)=\mathrm{Har}(B(L^2(M)),\mathcal{P}_\varphi)=\{ T\in B(L^2(M))\mid \mathcal{P}_\varphi(T)=T\}. $$
The \textbf{noncommutative Poisson boundary} $\mathcal{B}_\varphi$ of $M$ with respect to $\varphi$ is defined to be the noncommutative Poisson boundary of the u.c.p. map $\mathcal{P}_\varphi$ as defined by Izumi \cite{Izu02}, that is, the Poisson boundary $\mathcal{B}_\varphi$ is the unique $\mathrm{C}^*$-algebra (a von Neumann algebra when $\varphi$ is normal) that is isomorphic, as an operator system, to the space of harmonic operators $\mathrm{Har}(\mathcal{P}_\varphi)$. And the isomorphism $\CP:\CB_\varphi\to \mathrm{Har}(\mathcal{P}_\varphi)$ is called the $\varphi$-\textbf{Poisson transform}. Since $M\subset\mathrm{Har}(\mathcal{P}_\varphi)$, $M$ can also be embedded into $\CB_\varphi$ as a subalgebra. We said that $\CB_\varphi$ is trivial if $\CB_\varphi=M$.

According to \cite[Proposition 2.8]{DP20}, for a normal hyperstate $\varphi\in\mathcal{S}_\tau(B(L^2(M)))$, there exist a sequence $\{z_n\}\subset M$ such that $\sum_{n=1}^{\infty}z^*_nz_n=1$, and $\varphi$ and $\CP_\varphi$ admit the following standard form:
$$\varphi(T)=\sum_{n=1}^{\infty}\langle T  \hat{z}_n^*,\hat{z}_n^*\rangle, \ \mathcal{P}_\varphi(T)=\sum_{n=1}^{\infty} (Jz_n^*J)T(Jz_nJ), \ T\in B(L^2(M)).$$
Following \cite{DP20}, $\varphi$ is said to be 
\begin{itemize}
    \item \textbf{regular}, if $\sum_{n=1}^{\infty}\limits z_n^*z_n=\sum_{n=1}^{\infty}\limits z_nz_n^*=1$;
    \item \textbf{strongly generating}, if the unital algebra (rather than the unital $\ast$-algebra) generated by $\{z_n\}$ is weakly dense in $M$.
\end{itemize}
According to \cite{DP20}, when $\varphi$ is a normal regular strongly generating hyperstate, $\zeta:=\varphi\circ\CP\in\mathcal{S}_\tau(\CB_\varphi)$ is a normal faithful hyperstate on $\CB_\varphi$.

In this paper, we will focus on the Poisson boundaries of normal regular strongly generating hyperstates, which can be viewed as a generalization of admissible measures on countable discrete groups.  
\end{NCPB}

\begin{entropy}\label{NC entropy} Following \cite{DP20}, for a normal hyperstate $\varphi\in\mathcal{S}_\tau(B(L^2(M)))$, let $A_\varphi\in B(L^2(M))$ be the trace class operator associated to $\varphi$. The \textbf{entropy} of $\varphi$ is defined to be 
$$H(\varphi)=-\mathrm{Tr}(A_\varphi \log A_\varphi).$$
The \textbf{asymptotic entropy} of $\varphi$ is defined to be 
$$h(\varphi)=\lim_{n\to\infty}\frac{H(\varphi^{*n})}{n}.$$

For a von Neumann algebra $\CA$ such that $M\subset \CA$ and a normal faithful hyperstate $\zeta\in \mathcal{S}_\tau(\CA)$, let $\Delta_\zeta:L^2(\CA,\zeta)\to L^2(\CA,\zeta)$ be the modular operator of $(\CA,\zeta)$ and $e\in B(L^2(\CA,\zeta))$ be the orthogonal projection onto $L^2(M)$. The \textbf{Furstenberg-type entropy} of $(\CA,\zeta)$ with respect to $\varphi$ is defined to be 
$$h_\varphi(\CA,\zeta)=-\varphi(e\log\Delta_\zeta e).$$
\end{entropy}

\begin{Ultraproduct}\label{ultra} Following \cite{Oc85} and \cite{AH14}, for a sequence of $\sigma$-finite von Neumann algebras with normal faithful states $\{(\CA_n,\varphi_n)\}_{n\in \mathbb{N}}$ and an ultrafilter $\omega$ on $\mathbb{N}$, let 
\begin{equation}\label{UP}
\begin{aligned}
   l^\infty(\mathbb{N},\CA_n)&=\left\{(x_n)_n\in \prod_{n} \CA_n\mid\sup_n\Vert x_n \Vert<+\infty\right\},\\
   \mathcal{L}_\omega=\mathcal{L}_\omega(\CA_n,\varphi_n)&=\left\{(x_n)_n\in l^\infty(\mathbb{N},\CA_n)\mid\lim_\omega\varphi_n(x_n^*x_n)=0\right\},\\
   \mathcal{I}_\omega=\mathcal{I}_\omega(\CA_n,\varphi_n)&=\left\{(x_n)_n\in l^\infty(\mathbb{N},\CA_n)\mid\lim_\omega\varphi_n(x_n^*x_n+x_nx_n^*)=0\right\},\\
\mathcal{M}^\omega=\mathcal{M}^\omega(\CA_n,\varphi_n)&=\left\{(x_n)_n\in l^\infty(\mathbb{N},\CA_n)\mid(x_n)\mathcal{I}_\omega, \mathcal{I}_\omega(x_n)\subset\mathcal{I}_\omega\right\}.
\end{aligned}
\end{equation}
The \textbf{Ocneanu ultraproduct of $\{(\CA_n,\varphi_n)\}$ along $\omega$}, denoted by $(\CA_n,\varphi_n)^\omega$, is defined to be $(\CA^\omega, \varphi^\omega)=\mathcal{M}^\omega/\mathcal{I}_\omega$ the quotient 
$\mathrm{C}^*$-algebra with the quotient state, which can be proved to be a von Neumann algebra and a normal faithful state by the same proof as in \cite[\S 5.1]{Oc85}.

Let $\{(N_n,\pi_n,H_n)\}_{n\in \mathbb{N}}$ be a sequence of von Neumann algebras with normal faithful representation $\pi_n:N_n\to B(H_n)$. Let
$$\begin{aligned}
   l^\infty(\mathbb{N},H_n)&=\left\{(\xi_n)_n\in \prod_{n} H_n\mid\sup_n\Vert \xi_n \Vert<+\infty\right\},\\
   \mathcal{J}_\omega(\mathbb{N},H_n)&=\left\{(\xi_n)_n\in l^\infty(\mathbb{N},H_n)\mid\lim_\omega\Vert \xi_n \Vert=0\right\}.
\end{aligned}$$
The \textbf{ultraproduct of $(H_n)_n$ as Banach spaces} is defined to be
$$(H_n)_\omega:= l^\infty(\mathbb{N},H_n)/\mathcal{J}_\omega(\mathbb{N},H_n).$$
And $(H_n)_\omega$ is a Hilbert space with the inner product given by
$$\langle(\xi_n)_\omega,(\eta_n)_\omega\rangle=\lim_\omega\langle\xi_n,\eta_n\rangle, \ (\xi_n)_\omega,(\eta_n)_\omega\in (H_n)_\omega.$$
Let $(N_n)_\omega$ be the ultraproduct of $\{N_n\}_n$ as Banach spaces. Let $\pi_\omega:(N_n)_\omega\to B((H_n)_\omega)$ be the representation given by 
$$\pi_\omega((x_n)_\omega)((\xi_n)_\omega)=(x_n\xi_n)_\omega,\ (x_n)_n\in(N_n)_\omega,\ (\xi_n)_\omega\in (H_n)_\omega.$$
Then following \cite{Gr84}, \cite{Ra02} and \cite{AH14}, the \textbf{Groh-Raynaud ultraproduct of $(N_n,H_n)_n$ along $\omega$}, denoted by $\prod^\omega(N_n,H_n)$, is defined to be the weak closure of $\pi_\omega((N_n)_\omega)$ inside $B((H_n)_\omega)$.

Ando-Haagerup \cite{AH14} provided us with the relation between the two ultraproducts above: Let $(\CA_n,\varphi_n)$ be as above and $N:=\prod^\omega(\CA_n,L^2(\CA_n,\varphi_n))$ be the Groh-Raynaud ultraproduct. Let $\xi_{\varphi_n}\in L^2(\CA_n,\varphi_n)$ be the cyclic vector and $\xi_\omega=(\xi_{\varphi_n})_\omega\in(L^2(\CA_n,\varphi_n))_\omega$. Let $p\in N$ be the orthogonal projection onto $\overline{N' \xi_\omega}$ and define the normal state $\varphi_\omega$ on $N$ by $\varphi_\omega(x)=\langle x \xi_\omega,\xi_\omega\rangle$. 

Then according to \cite[Proposition 3.15]{AH14}, $\CA^\omega$ is $*$-isomorphic to $pNp$ with the isomorphism 
$$\theta:(x_n)^\omega\in \CA^\omega\mapsto p(x_n)_\omega p\in pNp.$$

Furthermore, $\theta: (\CA^\omega,\varphi^\omega)\to (pNp,\varphi_\omega|_{pNp})$ is also state preserving: By \cite[Lemma 3.13 and Proposition 3.14]{AH14}, since $\varphi_\omega|_{\mathcal{L}_\omega/\mathcal{I}_\omega+\mathcal{L}^*_\omega/\mathcal{I}_\omega}=0$, for any $x\in N$, we have
\begin{equation}\label{x pxp}
    \varphi_\omega(x)=\varphi_\omega(px)=\varphi_\omega(xp)=\varphi_\omega(pxp).
\end{equation}
Take $(x_n)_n\in \mathcal{M}^\omega$ and let $x=(x_n)_\omega$ in (\ref{x pxp}), we also have
\begin{equation}\label{x pxp2}
    \varphi_\omega(p(x_n)_\omega p)= \varphi_\omega((x_n)_\omega)=\langle (x_n)_\omega \xi_\omega,\xi_\omega\rangle=\lim_\omega\varphi_n(x_n)=\varphi^\omega((x_n)^\omega).
\end{equation}
Hence $\varphi_\omega(\theta((x_n)^\omega))=\varphi_\omega(p(x_n)_\omega p)=\varphi^\omega((x_n)^\omega)$. So $\theta: (\CA^\omega,\varphi^\omega)\to (pNp,\varphi_\omega|_{pNp})$ is a state preserving isomorphism.
\end{Ultraproduct}

\section{Entropy of noncommutative Poisson boundaries}
In classical Poisson boundary theory, for a countable discrete group $\Gamma$ and an admissible $\mu\in\prob(\Gamma)$ (i.e. $\cup_{n\in\mathbb{N}}(\supp \mu)^n=\Gamma$) with $H(\mu)<+\infty$, the following statements are Kaimanovich-Vershik's fundamental theorems regarding entropy (see \cite{KV82}):
\begin{itemize}
    \item [(a)] For any $(\Gamma,\mu)$-space $(X,\nu_X)$, one has $h_\mu(X,\nu_X)\leq h(\mu)$.
    \item [(b)] For the $\mu$-Poisson boundary $(B,\nu_B)$, the equality $h_\mu(B,\nu_B)= h(\mu)$ holds.
    \item [(c)] A $\mu$-boundary $(B_0,\nu_0)$ satisfies $h_\mu(B_0,\nu_0)= h_\mu(B,\nu_B)$ if and only if $(B_0,\nu_0)$ is measurably isomorphic to $(B,\nu_B)$.
    \item [(d)] The $\mu$-Poisson boundary $(B,\nu_B)$ is trivial (i.e. $B=\{*\}$ is a single point) if and only if $h(\mu)=0$.
\end{itemize}

In \cite{DP20}, Das-Peterson proved the noncommutative analogues of (a), a weaker version of (d) with $h(\mu)$ replaced by $h_\mu(B,\nu_B)$, and the ``if" direction of (d) as a direct corollary.

In this section, we will prove the noncommutative analogues of  (b)-(d). The method is inspired by \cite{SS22}, where Sayag-Shalom construct the classical Poisson boundaries through the ultraproduct of abelian $\mathrm{C}^*$-algebras and apply it to the proofs of (b) and some other results. For the Tomita-Takesaki theory involved in this section, we refer to \cite[Chapter VI-IX]{TakII} for details.

\begin{theorem}\label{ultra lim}
    Let $\varphi\in \mathcal{S}_\tau(B(L^2(M,\tau))$ be a normal regular strongly generating hyperstate, and $(\CB_\varphi,\zeta)$ be the Poisson boundary. Take an increasing sequence $\{a_n\}\subset\mathbb{R}$ satisfying $\frac{1}{2}\leq a_n< 1 $ and $\lim_n\limits a_n=1$, and an ultrafilter $\omega\in \beta \mathbb{N}\setminus \mathbb{N}$. Since $\varphi$ is strongly generating, we can take a sequence of von Neumann algebras with faithful normal hyperstates by 
    $$(\CA_n,\varphi_n)=(B(L^2(M)),(1-a_n)\sum_{k=0}^{\infty}a_n^k \varphi^{*k}),$$
    where $\varphi^{*0}(\,\cdot\,)=\langle\,\cdot\, \hat{1},\hat{1}\rangle$.
    Let $(\CA^\omega, \varphi^\omega)$ be the Ocneanu ultraproduct of $\{(\CA_n,\varphi_n)\}_n$ along $\omega$. Then $(\CB_\varphi,\zeta)$ can be embedded into $(\CA^\omega, \varphi^\omega)$ as a von Neumann subalgebra with normal conditional expectation and 
    $$h_\varphi(\CB_\varphi,\zeta)=h_\varphi(\CA^\omega, \varphi^\omega).$$
\end{theorem}

\begin{proof}
    Let $l^\infty(\mathbb{N},\CA_n)$, $\mathcal{L}_\omega$, $\mathcal{I}_\omega$ and $\mathcal{M}^\omega$ be as in (\ref{UP}). Then $\mathcal{L}_\omega$ is a left ideal of $l^\infty(\mathbb{N},\CA_n)$ and $(\CA^\omega, \varphi^\omega)=\mathcal{M}^\omega/\mathcal{I}_\omega$. For $(x_n)_n \in \mathcal{M}^\omega$, denote by $(x_n)^\omega\in \CA^\omega$ the image of $(x_n)_n $.

    Let $\Delta: B(L^2(M))\to l^\infty(\mathbb{N},\CA_n)$ be the diagonal embedding. First, let's prove that $\mathcal{L}_\omega\Delta(M)\subset \mathcal{L}_\omega$. 
     
    Assume that $\varphi=\sum_{k=1}^{\infty}\langle \, \cdot \,  \hat{z}_k^*,\hat{z}_k^*\rangle$ where $\sum_{k=1}^{\infty}z_kz_k^*=\sum_{k=1}^{\infty}z_k^*z_k=1$ and $\{z_k\}\subset M$ generates $M$ as a weakly closed unital subalgebra.
    
    Fix a $(x_n)_n \in \mathcal{L}_\omega$. Since
    \begin{equation}\label{phi_n phi}
    \varphi*\varphi_n=\frac{1}{a_n}\varphi_n-\frac{1-a_n}{a_n}\langle \, \cdot \,  \hat{1},\hat{1}\rangle    
    \end{equation}
    and $\lim_\omega\limits a_n=1$, we have $\lim_\omega\limits(\varphi*\varphi_n)(x^*_nx_n)=\lim_\omega\limits\varphi_n(x^*_nx_n)=0$. Note that for any $j\geq 1$,
    $$(\varphi*\varphi_n)(x^*_nx_n)=\sum_{k=1}^{\infty}\langle\mathcal{P}_{\varphi_n}(x^*_nx_n)\hat{z}_k^*,\hat{z}_k^*\rangle\geq \langle\mathcal{P}_{\varphi_n}(x^*_nx_n)\hat{z}_j^*,\hat{z}_j^*\rangle= \varphi_n (z_jx^*_nx_nz_j^*).$$
    Hence we also have $\lim_\omega\limits\varphi_n (z_jx^*_nx_nz_j^*) =0 $. Therefore, $\mathcal{L}_\omega\Delta(z^*_j)\subset\mathcal{L}_\omega$ for any $j\geq 1$. Let $\mathbb{C}[\{z^*_k\}]$ be the unital algebra generated by $\{z^*_k\}$. Then we have $\mathcal{L}_\omega \Delta(\mathbb{C}[\{z^*_k\}])\subset\mathcal{L}_\omega$. Hence for $z\in \mathbb{C}[(z^*_k)]$ and $(x_n)_n \in \mathcal{L}_\omega$, we have
    \begin{equation}\label{zxxz}
      \lim_\omega\limits\varphi_n (zx^*_nx_nz^*) =\lim_\omega\limits\Vert\CP_{\varphi_n}(x^*_nx_n)^{\frac{1}{2}}z^*\hat{1}\Vert^2=0.    
    \end{equation}
     
     Since $\varphi$ is strongly generating, $\mathbb{C}[\{z^*_k\}]$ is weakly dense in $M$, inducing that it is also strongly dense in $M$ (for a unital subalgebra, self-adjointness of its weak closure is enough to prove the von Neumann bicommutant theorem). So (\ref{zxxz}) also holds for any $z\in M$. Hence we have $\mathcal{L}_\omega\Delta(M)\subset \mathcal{L}_\omega$.

     Since $\mathcal{I}_\omega=\mathcal{L}_\omega \cap \mathcal{L}^*_\omega$ and $\mathcal{L}_\omega\Delta(M)\subset \mathcal{L}_\omega$, we have $\Delta(M)\subset \mathcal{M}^\omega$. Hence we can embed $(M,\tau)$ into $(\CA^\omega,\varphi^\omega)$ through the diagonal map: $x\in M\mapsto \Delta(x)+\mathcal{I}_\omega\in \CA^\omega$. And by (\ref{phi_n phi}), we know that $\varphi^\omega$ is $\varphi$-stationary.

     Let $N:=\prod^\omega(\CA_n,L^2(\CA_n,\varphi_n))$ be the Groh-Raynaud ultraproduct and define $\varphi_\omega$ and $p$ as in \ref{ultra}. As in \ref{ultra}, we can identify $(\CA^\omega,\varphi^\omega)$ with $(pNp,\varphi_\omega|_{pNp})$ through the state preserving isomorphism $$\theta:(x_n)^\omega\in \CA^\omega\mapsto p(x_n)_\omega p\in pNp.$$

    Let $\Delta_0:(\har(\CP_\varphi),\varphi)\to (N,\varphi_\omega)$ be the diagonal map. Since $\varphi_n|_{\har(\CP_\varphi)}=\varphi|_{\har(\CP_\varphi)}$, we know that for $T\in\har(\CP_\varphi)$,
    \begin{equation}\label{phi D_0}
        \varphi_\omega(\Delta_0(T))=\lim_\omega \varphi_n(T)=\varphi(T).
    \end{equation}
      
    Define $\Bar{\Delta}: (\har(\CP_\varphi),\varphi)\to (\CA^\omega,\varphi^\omega)$ by 
      $$\Bar{\Delta}(T)=\theta^{-1}(p\Delta_0(T)p),\ T\in \har(\CP_\varphi).$$
    Then by (\ref{x pxp}), (\ref{phi D_0}) and $\varphi_\omega|_{pNp}=\varphi^\omega\circ\theta^{-1}$, for $T\in \har(\CP_\varphi)$, we have
      $$\varphi(T)=\varphi_\omega(\Delta_0(T))=\varphi_\omega(p\Delta_0(T)p)=\varphi^\omega(\theta^{-1}(p\Delta_0(T)p))=\varphi^\omega(\Bar{\Delta}(T)).$$ 
    Hence $\Bar{\Delta}: (\har(\CP_\varphi),\varphi)\to (\CA^\omega,\varphi^\omega)$ is state preserving. But the normality of $\Bar{\Delta}$ is unknown yet. Also note that for $x\in M$, 
      $$\Bar{\Delta}(x)=\theta^{-1}(p(x)_\omega p)=(x)^\omega=\Delta(x)+\mathcal{I}_\omega.$$
    Hence $\Bar{\Delta}|_M$ is exactly the natural embedding: $x\in M\mapsto \Delta(x)+\mathcal{I}_\omega\in \CA^\omega$. Moreover, by \cite[Theorem 3.1]{Ch74} (although $\har(\CP_\varphi)$ is not a $\mathrm{C}^*$-algebra, the same proof works), $M$ is in the multiplicative domain of $\Bar{\Delta}$, that is, $$\Bar{\Delta}(aTb)=\Bar{\Delta}(a)\Bar{\Delta}(T)\Bar{\Delta}(b)$$ 
    for any $a, b\in M$ and $T\in \har(\CP_\varphi)$.

    Let's prove that $\CP_{\varphi^\omega}(\Bar{\Delta}(T))=T$ for any $T\in\har(\CP_\varphi)$. Only need to prove for any $a,b\in M$, we have $\varphi^\omega(\Bar{\Delta}(a)\Bar{\Delta}(T)\Bar{\Delta}(b))=\langle aTb\hat{1},\hat{1}\rangle$.

    Identifying $(\CA^\omega,\varphi^\omega)$ with $(pNp,\varphi_\omega|_{pNp})$, we have 
$$\varphi^\omega(\Bar{\Delta}(a)\Bar{\Delta}(T)\Bar{\Delta}(b))=\varphi^\omega(\Bar{\Delta}(aTb))=\varphi_\omega(p\Delta_0(aTb)p).$$
Moreover, by (\ref{x pxp}), we have
$$\varphi^\omega(\Bar{\Delta}(a)\Bar{\Delta}(T)\Bar{\Delta}(b))=\varphi_\omega(p\Delta_0(aTb)p)=\varphi_\omega(\Delta_0(aTb))=\lim_\omega \varphi_n(aTb)=\langle aTb\hat{1},\hat{1}\rangle.$$
    The last equality holds because $aTb\in \har(\CP_\varphi)\subset \har(\CP_{\varphi_n})$. Hence we must have $\CP_{\varphi^\omega}(\Bar{\Delta}(T))=T$ for any $T\in \har(\CP_\varphi)$.

    Let $(\CB,\varphi^\omega):=(\Bar{\Delta}(\har(\CP_\varphi)),\varphi^\omega)\subset(\CA^\omega,\varphi^\omega)$. Now let's prove that $(\CB,\varphi^\omega)$ is a von Neumann subalgebra isomorphic to $(\CB_\varphi,\zeta)$. 

    Since $\varphi^\omega$ is $\varphi$-stationary, we have $\CP_{\varphi^\omega}(\CA^\omega)\subset\har(\CP_\varphi)$. So we can define 
    $$E:=\Bar{\Delta}\circ \CP_{\varphi^\omega}:\CA^\omega\to \CA^\omega.$$ 
    Since $\CP_{\varphi^\omega}\circ\Bar{\Delta}=\id_{\har(\CP_\varphi)}$, we have $E^2=E$. So $E$ is a conditional expectation from $\CA^\omega$ onto $\Bar{\Delta}(\har(\CP_\varphi))$. Note that $E$ also preserves the faithful state $\varphi^\omega$. According to a well-known result \cite[Lemma 2.6]{DP20}, which also appears in earlier literature like \cite{FNW94}, \cite[Lemma 3.4]{BJKW00} and \cite[Lemma 3.1]{CD20}, $\CB=\Bar{\Delta}(\har(\CP_\varphi))=\har(E)(=\{T\in \CA^\omega\mid E(T)=T\})$ is a $\mathrm{C}^*$-subalgebra of $\CA^\omega$.

    Since for any $(T_{ij})_{n\times n}\in M_n(\har(\CP_\varphi))$, 
    $$\Vert (T_{ij})\Vert=\Vert \CP_{\varphi^\omega}(\Bar{\Delta}((T_{ij})))\Vert\leq\Vert \Bar{\Delta}((T_{ij}))\Vert \leq \Vert (T_{ij})\Vert,$$
    we have that both $\Bar{\Delta}$ and $\CP_{\varphi^\omega}|_\CB $ are completely isometric. Therefore, $(\CB,\varphi^\omega)$ is also a Poisson boundary of $\varphi$ and $\CP_{\varphi^\omega}|_\CB:(\CB,\varphi^\omega)\to (\har(\CP_\varphi),\varphi)$ is the Poisson transform.

    Let $\mathcal{P}:(\CB_\varphi,\zeta)\to (\har(\CP_\varphi),\varphi)$ be the Poisson transform. Define $\Phi:(\CB_\varphi,\zeta)\to(\CB,\varphi^\omega)$ by $\Phi=\Bar{\Delta}\circ\mathcal{P}=(\CP_{\varphi^\omega}|_\CB)^{-1}\circ\mathcal{P}$. Following \cite[Remark 3.4]{Ch74}, a bijective u.c.p. map between $\mathrm{C}^*$-algebras with a u.c.p. inverse is necessarily a $\mathrm{C}^*$-isomorphism. Hence $\Phi:\CB_\varphi\to\CB$ is a $\mathrm{C}^*$-isomorphism. Since both $\Bar{\Delta}$ and $\CP$ are state preserving, we know that $\Phi:(\CB_\varphi,\zeta)\to(\CA^\omega,\varphi^\omega)$ preserves the normal faithful state $\zeta$. Hence $\Phi:(\CB_\varphi,\zeta)\to(\CA^\omega,\varphi^\omega)$ is normal, $\CB=\Phi(\CB_\varphi)$ is a von Neumann subalgebra of $\CA^\omega$ and $\Phi:\CB_\varphi\to\CB$ is a $\mathrm{W}^*$-isomorphism that embeds $(\CB_\varphi,\zeta)$ into $(\CA^\omega,\varphi^\omega)$. And by \cite[Theorem IX.4.2]{TakII}, the conditional expectation $E:\CA^\omega\to \CB$ is normal.

    Now let's prove that $h_\varphi(\CB_\varphi,\zeta)=h_\varphi(\CA^\omega, \varphi^\omega)$. By the isomorphism, we only need to prove that $h_\varphi(\CB,\varphi^\omega)=h_\varphi(\CA^\omega, \varphi^\omega)$.

    Note that we have $(M,\tau)\subset(\CB,\varphi^\omega)\subset(\CA^\omega,\varphi^\omega)$ by identifying $M$ with $\Bar{\Delta}(M)$. According to \cite[Lemma 5.4]{DP20}, for $(M,\tau)\subset(\mathcal{A},\nu)$ we have
    \begin{equation}\label{Deltait}
    h_\varphi(\mathcal{A},\nu)=i\sum_{k=1}^{\infty}\left\langle \lim_{t\to0} \frac{\Delta_{(\mathcal{A},\nu)}^{it}-1}{t}\hat{z}_k^*,\hat{z}_k^*\right\rangle,
    \end{equation}
    where $\Delta_{(\mathcal{A},\nu)}$ is the modular operator of $(\mathcal{A},\nu)$. 

    By (\ref{Deltait}) we only need to prove $e_M\Delta_{(\CB,\varphi^\omega)}^{it}e_M=e_M\Delta_{(\CA^\omega,\varphi^\omega)}^{it}e_M$, where $e_M$ is the orthogonal projection from $L^2(\CA^\omega,\varphi^\omega)$ to $L^2(M,\tau)$.

   Recall that $E=\Bar{\Delta}\circ \CP_{\varphi^\omega}:\CA^\omega\to \CB$ is a normal conditional expectation. According to \cite[Theorem IX.4.2]{TakII}, we have $\Delta_{(\CB,\varphi^\omega)}^{it}=\Delta_{(\CA^\omega,\varphi^\omega)}^{it}|_{L^2(\CB,\varphi^\omega)}$.

   Hence $e_M\Delta_{(\CB,\varphi^\omega)}^{it}e_M=e_Me_\CB\Delta_{(\CA^\omega,\varphi^\omega)}^{it}e_\CB e_M=e_M\Delta_{(\CA^\omega,\varphi^\omega)}^{it}e_M$ and
   $$h_\varphi(\CB,\varphi^\omega)=i\sum_{k=1}^{\infty}\left\langle \lim_{t\to0} \frac{\Delta_{(\CB,\varphi^\omega)}^{it}-1}{t}\hat{z}_k^*,\hat{z}_k^*\right\rangle=i\sum_{k=1}^{\infty}\left\langle \lim_{t\to0} \frac{\Delta_{(\CA^\omega,\varphi^\omega)}^{it}-1}{t}\hat{z}_k^*,\hat{z}_k^*\right\rangle=h_\varphi(\CA^\omega,\varphi^\omega).$$

    Therefore, we have
    $$h_\varphi(\CB_\varphi,\zeta)=h_\varphi(\CB ,\varphi^\omega)=h_\varphi(\CA^\omega,\varphi^\omega).$$
\end{proof}

\begin{lemma}\label{1+t}
   With the same notation as above, when $n=\omega$, define $(\CA_n,\varphi_n)$ to be $(\CA^\omega,\varphi^\omega)$. For $n\in \mathbb{N}\cup \{\omega\}$, let $\Delta_{\varphi_n}$ be the modular operator of $(\CA_n,\varphi_n)$, $e_n$ be the orthogonal projection from $L^2(\CA_n, \varphi_n)$ to $L^2(M,\tau)$ and $A_{\varphi} \in B(L^2(M))$ be the positive trace class operator associated to ${\varphi}$. Note that $B(L^2(M))= e_n B(L^2(\CA_n,\varphi_n))e_n$. Then for any $n\in \mathbb{N}\cup \{\omega\}$ and $t>0$, within $B(L^2(M))$, we have 
   $$(1+t)^{-1} \leq e_n(\Delta_{\varphi_n}+t)^{-1}e_n\leq (\frac{1}{2}A_{\varphi}+t)^{-1}.$$
   \end{lemma}

\begin{proof}
    Let $\{E(\lambda)\}_{\lambda\geq 0}$ be the spectral measure of $\Delta_{\varphi_n}$. Since $e_n\Delta_{\varphi_n}e_n=\Delta_{(M,\tau)}=1$, for any $z\in L^2(M,\tau)$ such that $\Vert z \Vert_\tau=1$, we have 
    $$
    \begin{aligned}
    \langle e_n(\Delta_{\varphi_n}+t)^{-1}e_n z,z\rangle=&\int_0^{+\infty}(\lambda+t)^{-1}\d \langle E(\lambda)z,z\rangle\\
    \geq &\left(\int_0^{+\infty}\lambda\d \langle E(\lambda)z,z\rangle+t\right)^{-1}\\
    =&(\langle \Delta_{\varphi_n}z,z\rangle+t)^{-1}\\
    =&(\langle e_n\Delta_{\varphi_n}e_nz,z\rangle+t)^{-1}\\
    =&(1+t)^{-1}\\
    =&\langle (1+t)^{-1}z,z\rangle.
    \end{aligned}
    $$
    The second inequality holds for the convexity of $x\mapsto(x+t)^{-1}$. So now we have $ e_n(\Delta_{\varphi_n}+t)^{-1}e_n\geq (1+t)^{-1}$ within $B(L^2(M))$.

    Now let's prove $e_n(\Delta_{\varphi_n}+t)^{-1}e_n\leq (\frac{1}{2}A_\varphi+t)^{-1}$. First, we prove that $\Delta_{\varphi_n}\geq \frac{1}{2}T^*A_\varphi T $ on $D(\Delta_{\varphi_n})$, where $T:L^2(\CA_n,\varphi_n)\to L^2(M,\tau)$ is defined by $T(a \xi_{\varphi_n})=\mathcal{P}_{\varphi_n}(a)\hat{1} \ (a\in \CA_n)$ and $\xi_{\varphi_n}$ is the cyclic vector of $L^2(\CA_n,\varphi_n)$.

    When $n=\omega$, since $\varphi^\omega$ is $\varphi$-stationary, according to the proof of \cite[Lemma 5.12]{DP20}, we have $\Delta_{\varphi_n}\geq T^*A_\varphi T\geq \frac{1}{2}T^*A_\varphi T$ on $D(\Delta_{\varphi_n})$.

    When $n\in \mathbb{N}$, following the same discussion of the proof of \cite[Lemma 5.12]{DP20}, we still have that for $a \in \CA_n$, 
    $$\Vert \Delta_{\varphi_n}^{1/2} a \xi_{\varphi_n}\Vert^2=\langle\mathcal{P}_{\varphi_n}(aa^*)\hat{1},\hat{1}\rangle, $$
    $$\Vert A_{\varphi}^{1/2}T a \xi_{\varphi_n}\Vert^2\leq\langle(\mathcal{P}_{\varphi}\circ\mathcal{P}_{\varphi_n})(aa^*)\hat{1},\hat{1}\rangle.$$
    Since $\varphi*\varphi_n=\frac{1}{a_n}\varphi_n-\frac{1-a_n}{a_n}\langle \, \cdot \,  \hat{1},\hat{1}\rangle$, we have $\mathcal{P}_{\varphi}\circ\mathcal{P}_{\varphi_n}=\frac{1}{a_n} \mathcal{P}_{\varphi_n}-\frac{1-a_n}{a_n}\mathrm{id}$. Therefore,
    $$\Vert A_{\varphi}^{1/2}T a \xi_{\varphi_n}\Vert^2\leq\langle(\mathcal{P}_{\varphi}\circ\mathcal{P}_{\varphi_n})(aa^*)\hat{1},\hat{1}\rangle \leq \frac{1}{a_n}\langle\mathcal{P}_{\varphi_n}(aa^*)\hat{1},\hat{1}\rangle\leq 2 \Vert \Delta_{\varphi_n}^{1/2} a \xi_{\varphi_n}\Vert^2.$$
    Hence $\Delta_{\varphi_n}\geq \frac{1}{2}T^*A_\varphi T $ on $D(\Delta_{\varphi_n})$.

    Now we have $\Delta_{\varphi_n}\geq \frac{1}{2}T^*A_\varphi T $ for $n\in \mathbb{N}\cup\{\omega\}$, which induces $(\Delta_{\varphi_n}+t)^{-1}\leq(\frac{1}{2}T^*A_{\varphi}T+t)^{-1}$ within $B(L^2(\CA_n,\varphi_n))$. 

    Since $A_\varphi$ is a positive trace class operator, we can take $\{b_k\}$ to be an orthogonal normalized basis of $L^2(M)$ formed by eigenvectors of $A_\varphi$. Assume that $A_\varphi b_k=\lambda_k b_k \ (\lambda_k\geq 0)$. Since $T|_{L^2(M)}=T^*|_{L^2(M)}=\id$, we have $T^*A_{\varphi}Tb_k=\lambda_k b_k$. So $\{b_k\}$ are also the eigenvectors of $T^*A_{\varphi}T$. Hence
    $$(\frac{1}{2}T^*A_{\varphi}T+t)^{-1}b_k=(\frac{1}{2}\lambda_k+1)^{-1}b_k= (\frac{1}{2}A_{\varphi}+t)^{-1}b_k.$$
    Therefore, $e_n(\frac{1}{2}T^*A_{\varphi}T+t)^{-1}e_n=(\frac{1}{2}A_{\varphi}+t)^{-1}$ and 
    $$e_n(\Delta_{\varphi_n}+t)^{-1}e_n\leq e_n(\frac{1}{2}T^*A_{\varphi}T+t)^{-1}e_n = (\frac{1}{2}A_{\varphi}+t)^{-1}.$$

\end{proof}

Now we are ready to prove the analogue of the fact that for the $(\Gamma,\mu)$-Poisson boundary $(B,\nu_B)$, $h_\mu(B,\nu_B)=h(\mu)$.

\begin{theorem}
    With the same notation as above, assume that $H(\varphi)<+\infty$. Then we have 
    $$h_\varphi(\CB_\varphi,\zeta)=h_\varphi(\CA^\omega, \varphi^\omega)=\lim_\omega\limits h_\varphi(\CA_n, \varphi_n)=h(\varphi).$$
\end{theorem}

\begin{proof}

    According to \cite[Corollary 5.13]{DP20} and Theorem \ref{ultra lim}, we already have $$h(\varphi)\geq h_\varphi(\CB_\varphi,\zeta)=h_\varphi(\CA^\omega, \varphi^\omega).$$ Only need to prove $$h_\varphi(\CA^\omega, \varphi^\omega)=\lim_\omega\limits h_\varphi(\CA_n, \varphi_n)\geq h(\varphi).$$

    Inspired by the proof of \cite[Lemma 5.14]{DP20}, we use the integral representation of $\log x$ given by 
    $$\log x=\int_{0}^{+\infty}[(1+t)^{-1}-(x+t)^{-1}]\d t.$$
    For $n\in \mathbb{N}\cup\{\omega\}$, let $e_n$ be the orthogonal projection from $L^2(\CA_n,\varphi_n)$ onto $L^2(M)$, then we have
    \begin{equation}\label{int varphi}
     h_\varphi(\CA_n, \varphi_n)=-\varphi(e_n\log \Delta_{\varphi_n}e_n)=\int_{0}^{+\infty}\varphi(e_n[(\Delta_{\varphi_n}+t)^{-1}-(1+t)^{-1}]e_n)\d t.   
    \end{equation}
    The final equality holds because according to Lemma \ref{1+t}, everything is positive and we can apply Fubini's theorem.

    For $n\in \mathbb{N}\cup\{\omega\}$, define $F_n:(0,+\infty)\to \mathbb{R}$ by
    $$F_n(t)=\varphi(e_n[(\Delta_{\varphi_n}+t)^{-1}-(1+t)^{-1}]e_n).$$
    Then $h_\varphi(\CA_n,\varphi_n)=\int_{0}^{+\infty} F_n(t) \, \d t.$
    
    Define $G:(0,+\infty)\to \mathbb{R}$ by
    $$G(t)=\varphi(( \frac{1}{2}A_{\varphi}+t)^{-1}-(1+t)^{-1}).$$
    Then by Lemma \ref{1+t}, we have $0\leq F_n(t)\leq G(t)$. 

    Since
    $$
    \begin{aligned}
        \int_{0}^{+\infty}G(t)\,\d t&=\varphi\left(\int_{0}^{+\infty}[( \frac{1}{2}A_{\varphi}+t)^{-1}-(1+t)^{-1}]\,\d t\right)=-\varphi(\log  (\frac{1}{2}A_\varphi))\\
        &=-\varphi(\log A_\varphi)+\log 2=-\mathrm{Tr}(A_\varphi \log A_\varphi)+\log 2=H(\varphi)+\log 2<+\infty,
    \end{aligned}$$
    we have $G\in L^1((0,+\infty))$.

    Let's prove that for $t>0$, $\lim_\omega\limits F_n(t)=F_\omega(t)$. According to the proof of \cite[Theorem 4.1]{AH14}, we have that for $(x_n)^\omega\in \CA^\omega$, 
    $$
     (\Delta_{\varphi^\omega}+1)^{-1}((x_n)^\omega \xi_{\varphi^\omega})=((\Delta_{\varphi_n}+1)^{-1}(x_n\xi_{\varphi_n}))_\omega.   
    $$
    For any $t>0$, consider the functional calculation about $f\in C([0,1])$ satisfying $f(s)=\frac{s}{1-(1-t)s}$, we have
    \begin{equation*}
     f((\Delta_{\varphi^\omega}+1)^{-1})((x_n)^\omega \xi_{\varphi^\omega})=(f((\Delta_{\varphi_n}+1)^{-1})(x_n\xi_{\varphi_n}))_\omega.   
    \end{equation*}
    That is
    \begin{equation}\label{D+t}
     (\Delta_{\varphi^\omega}+t)^{-1}((x_n)^\omega \xi_{\varphi^\omega})=((\Delta_{\varphi_n}+t)^{-1}(x_n\xi_{\varphi_n}))_\omega.   
    \end{equation}
    For any $x,y \in M$ and $t>0$, by (\ref{D+t}) we have
    $$\langle (\Delta_{\varphi^\omega}+t)^{-1}\Bar{\Delta}(x)\xi_{\varphi^\omega},\Bar{\Delta}(y)\xi_{\varphi^\omega}\rangle=\lim_\omega \langle (\Delta_{\varphi_n}+t)^{-1}x\hat{1},y\hat{1}\rangle.$$
    That is 
    $$e(\Delta_{\varphi^\omega}+t)^{-1}e=\lim_\omega\limits e_n(\Delta_{\varphi_n}+t)^{-1}e_n$$
    with respect to ultraweak topology in $B(L^2(M))$, where $e$ is the orthogonal projection from $L^2(\CA^\omega,\varphi^\omega)$ onto $L^2(M)$. 

    Since $\varphi$ is normal, we have
    $$ F_\omega(t)=\varphi(e(\Delta_{\varphi^\omega}+t)^{-1}e)=\lim_\omega\limits \varphi(e_n(\Delta_{\varphi_n}+t)^{-1}e_n)=\lim_\omega F_n(t).$$

    Now we have $F_n$ converge to $F_\omega$ pointwise when $n\to \omega$ and they are uniformly dominated by $G\in L^1((0,+\infty))$. But Lebesgue's dominated convergence theorem doesn't work for nets, we need further discussion to prove $\int_{0}^{+\infty} F_\omega(t) \, \d t=\lim_\omega\limits \int_{0}^{+\infty} F_n(t) \, \d t$.

    Let's prove that $\{F_n\}$ is uniformly equicontinuous on any compact subset of $(0,+\infty)$. For any $[\epsilon,1/\epsilon]\subset(0,+\infty)$ and $t_1,t_2\in [\epsilon,1/\epsilon]$, we have

    $$
    \begin{aligned}
        |F_n(t_1)-F_n(t_2)|=&|\varphi(e_n[(\Delta_{\varphi_n}+t_1)^{-1}-(\Delta_{\varphi_n}+t_2)^{-1}]e_n)|\\
        \leq& \Vert(\Delta_{\varphi_n}+t_1)^{-1}-(\Delta_{\varphi_n}+t_2)^{-1}\Vert \\
        \leq & \Vert (x+t_1)^{-1}-(x+t_2)^{-1}\Vert_{L^\infty((0,+\infty))}\\
        =&\frac{|t_1-t_2|}{t_1t_2}\\
        \leq & \epsilon^{-2}|t_1-t_1|.
    \end{aligned}
    $$
    Hence $\{F_n\}$ is uniformly equicontinuous on any compact subset of $(0,+\infty)$. Together with the pointwise convergence, we know that when $n\to \omega$, $F_n$ uniformly converge to $F_\omega$ on any compact subset of $(0,+\infty)$ and dominated by $G\in L^1((0,+\infty))$ at the same time. Therefore,
    $$\int_{0}^{+\infty} F_\omega(t) \, \d t=\lim_\omega\limits \int_{0}^{+\infty} F_n(t) \, \d t,$$
    which is equivalent to $h_\varphi(\CA^\omega, \varphi^\omega)=\lim_\omega\limits h_\varphi(\CA_n, \varphi_n) $.

    Finally, let's prove $\lim_\omega\limits h_\varphi(\CA_n, \varphi_n)\geq h(\varphi)$. According to \cite[Example 5.5]{DP20} (where there is an obvious typo), we have
    $$
    \begin{aligned}
    h_\varphi(\CA_n, \varphi_n)=&h_\varphi(B(L^2(M)), \varphi_n)\\
    =&-\mathrm{Tr}(A_{\varphi*\varphi_n}\log A_{\varphi_n})+\langle \CP_{\varphi_n}(\log A_{\varphi_n}) \hat{1},\hat{1}\rangle \\
    =& -\mathrm{Tr}(A_{\varphi*\varphi_n}\log A_{\varphi_n})+\mathrm{Tr}(A_{\varphi_n}\log A_{\varphi_n})
    \end{aligned}
    $$
    Since $\varphi*\varphi_n=\frac{1}{a_n}\varphi_n-\frac{1-a_n}{a_n}\langle \, \cdot \,  \hat{1},\hat{1}\rangle$, we have $A_{\varphi*\varphi_n}=\frac{1}{a_n}A_{\varphi_n}-\frac{1-a_n}{a_n}P_{\hat{1}}$, where $P_{\hat{1}}$ is the orthogonal projection onto $\mathbb{C}\cdot\hat{1}$.

    Hence
    $$h_\varphi(\CA_n, \varphi_n)=-\frac{1-a_n}{a_n}\mathrm{Tr}(A_{\varphi_n}\log A_{\varphi_n})+\frac{1-a_n}{a_n}\mathrm{Tr}(P_{\hat{1}}\log A_{\varphi_n})=\frac{1-a_n}{a_n}H(\varphi_n)+\frac{1-a_n}{a_n}\langle \log A_{\varphi_n} \hat{1},\hat{1}\rangle.$$
    Since $\varphi_n=(1-a_n)\sum_{k=0}^{\infty}\limits a^k_n\varphi^{*k}$, we have 
    $$A_{\varphi_n}=(1-a_n)\sum_{k=0}^{\infty}\limits a^k_nA_{\varphi^{*k}}\geq (1-a_n)P_{\hat{1}}.$$
    By L\"owner-Heinz theorem (\cite{Lo34} and \cite{He51}, see also \cite[Chapter V]{Bh97}), $x\mapsto \log x$ is operator monotone. Therefore, for any $\epsilon>0$, we have 
    $$\frac{1-a_n}{a_n}\langle \log (A_{\varphi_n}+\epsilon) \hat{1},\hat{1}\rangle\geq \frac{1-a_n}{a_n}\langle \log [(1-a_n)P_{\hat{1}}+\epsilon] \hat{1},\hat{1}\rangle=\frac{1-a_n}{a_n}\log(1-a_n+\epsilon)>\frac{1-a_n}{a_n}\log(1-a_n).$$
    Let $\{E_n(\lambda)\}_{\lambda\geq 0}$ be the spectral measure of $A_{\varphi_n}$. By monotone convergence theorem, we have
    $$\langle \log (A_{\varphi_n}+\epsilon) \hat{1},\hat{1}\rangle=\int_0^{+\infty}\log (\lambda+\epsilon)\d\langle E_n(\lambda)\hat{1},\hat{1}\rangle\to\int_0^{+\infty}\log \lambda\,\d\langle E_n(\lambda)\hat{1},\hat{1}\rangle=\langle \log A_{\varphi_n} \hat{1},\hat{1}\rangle \ (\epsilon\to 0).$$
    Therefore, 
    $$\frac{1-a_n}{a_n}\langle \log A_{\varphi_n} \hat{1},\hat{1}\rangle=\lim_{\epsilon\to 0+}\frac{1-a_n}{a_n}\langle \log (A_{\varphi_n}+\epsilon) \hat{1},\hat{1}\rangle\geq \frac{1-a_n}{a_n}\log(1-a_n)\to 0 \ (n\to\omega).$$

    By L\"owner-Heinz theorem regarding operator convexity (\cite{Kr36} and \cite{BS55}) or more directly, \cite[Exercise V.2.13]{Bh97}, $x\mapsto -x\log x$ is operator concave. By the definition $H(\psi)=-\mathrm{Tr}(A_\psi\log A_\psi)$, we know that $H(\,\cdot\,)$ is concave. Also, according to \cite[Theorem 5.1]{DP20}, $H(\varphi^{*k})\geq kh(\varphi)$. Therefore,
    $$\frac{1-a_n}{a_n}H(\varphi_n)\geq \frac{(1-a_n)^2}{a_n}\sum_{k=0}^{\infty}a_n^kH(\varphi^{*k})\geq \frac{(1-a_n)^2}{a_n}\sum_{k=0}^{\infty}a_n^k\cdot k\cdot h(\varphi)=\frac{1}{a_n}h(\varphi)\geq h(\varphi).$$
    Hence
    $$\lim_\omega h_\varphi(\CA_n, \varphi_n)=\lim_\omega \frac{1-a_n}{a_n}H(\varphi_n)+\frac{1-a_n}{a_n}\langle \log A_{\varphi_n} \hat{1},\hat{1}\rangle\geq h(\varphi)+0=h(\varphi),$$
    which finishes the proof.
\end{proof}

Recall that in the classical theory, for a quasi-invariant $\Gamma$-space $(X,\nu_X)$, the Radon-Nikodym factor of $(X,\nu_X)$ is a quasi-invariant $\Gamma$-space $(X_\mathrm{RN},\Sigma_\mathrm{RN},\nu_{X_\mathrm{RN}})$, where $\Sigma_\mathrm{RN}$ is generated by $\{\frac{\d \gamma\nu_X}{\d\nu_X}\}_{\gamma\in\Gamma}$, $X_\mathrm{RN}$ is the topological model of $\Sigma_\mathrm{RN}$ and $\nu_{X_\mathrm{RN}}=\nu_X|_{\Sigma_\mathrm{RN}}$. The following defines a generalization of Radon-Nikodym factor.

\begin{definition}\label{RN factor}
    Let $(\mathcal{A},\nu)$ be a von Neumann algebra with normal faithful state such that $(M,\tau)\subset(\mathcal{A},\nu)$. We define the \textbf{noncommutative Radon-Nikodym factor} of $(\mathcal{A},\nu)$ to be $(\mathcal{A},\nu)_{\mathrm{RN}}:=(\mathcal{A}_{\mathrm{RN}},\nu|_{\mathcal{A}_{\mathrm{RN}}})$, where $\mathcal{A}_{\mathrm{RN}}$ is the von Neumann subalgebra of $\mathcal{A}$ generated by $\{\sigma^\nu_t(x)|x\in M,t\in\mathbb{R}\}$ and $\{\sigma^\nu_t(x)\}_{t\in \mathbb{R}}$ is the modular automorphism group of $(\mathcal{A},\nu)$.
\end{definition}

The following example shows that the noncommutative Radon-Nikodym factor is a generalization of the classical one.

\begin{example}
    Let $\Gamma$ be a countable discrete group and $(X,\nu_X)$ be a quasi-invariant $\Gamma$-space. Let $(L(\Gamma),\tau)$ be the group von Neumann algebra of $\Gamma$ with canonical trace. Define the $\tau$-hyperstate $\nu$ on $L(\Gamma \curvearrowright X)$ by $\nu(\sum_{\gamma\in\Gamma}a_\gamma u_\gamma)=\int a_e \d \nu_X$. Then following \cite[Theorem X.1.17]{TakII}, the modular operator $\Delta_\nu$ is 
    $$\Delta_\nu(\gamma,x)=\frac{\d \gamma^{-1}\nu_X}{\d\nu_X}(x).$$
    And the modular automorphism group satisfies 
    $$\sigma_t^\nu(u_\gamma)=\left(\frac{\d \gamma^{-1}\nu_X}{\d\nu_X}\right)^{it}u_\gamma.$$
    Hence the $(L(\Gamma),\tau)$-Radon-Nikodym factor of $(L(\Gamma \curvearrowright X),\nu)$ is generated by $\left\{\left(\frac{\d \gamma^{-1}\nu_X}{\d\nu_X}\right)^{it}u_\gamma\right\}_{\gamma\in\Gamma,t\in\mathbb{R}}$, which is exactly the group measure space construction $L(\Gamma \curvearrowright X_\mathrm{RN})$.
\end{example}

\begin{theorem}\label{RN=RN}
   Let $\varphi\in \mathcal{S}_\tau(B(L^2(M,\tau))$ be a normal regular strongly generating hyperstate. Let $(\mathcal{A},\nu)$ be as in Definition \ref{RN factor} such that $h_\varphi(\mathcal{A},\nu)<+\infty$. Let $(\mathcal{A}_0,\nu_0)$ be a von Neumann subalgebra of $(\mathcal{A},\nu)$ such that $(M,\tau)\subset(\mathcal{A}_0,\nu_0)$. Then $h_\varphi(\mathcal{A}_0,\nu_0)=h_\varphi(\mathcal{A},\nu)$ if and only if $(\mathcal{A}_0,\nu_0)_{\mathrm{RN}}=(\mathcal{A},\nu)_{\mathrm{RN}}$.
\end{theorem}

\begin{proof}
   Denote by $\Delta_0$ and $\Delta$ the modular operators of $(\mathcal{A}_0,\nu_0)$ and $(\mathcal{A},\nu)$; denote by $\{\sigma_t^0\}_{t\in \mathbb{R}}$ and $\{\sigma_t\}_{t\in \mathbb{R}}$ the modular automorphism groups of $(\mathcal{A}_0,\nu_0) $ and $(\mathcal{A},\nu)$. Then $\Delta_0=e_0\Delta e_0$, where $e_0$ is the orthogonal projection from $L^2(\mathcal{A},\nu)$ onto $L^2(\mathcal{A}_0,\nu_0)$. 

   Assume that $(\mathcal{A}_0,\nu_0)_{\mathrm{RN}}=(\mathcal{A},\nu)_{\mathrm{RN}}$. Denote by $\{\sigma_t^{\mathrm{RN}}\}_{t\in \mathbb{R}}$ the modular automorphism group of $(\mathcal{A},\nu)_\mathrm{RN}$. Since $\mathcal{A}_\mathrm{RN}$ is invariant under the action of $\{\sigma_t\}_{t\in \mathbb{R}}$, by Tomita-Takesaki theory \cite[Theorem IX.4.2 and Theorem IX.4.18]{TakII}, we have $\sigma^{\mathrm{RN}}_t=\sigma_t|_{\mathcal{A}_\mathrm{RN}}$ for any $t\in\mathbb{R}$. By (\ref{Deltait}), we have $h_\varphi(\mathcal{A}_\mathrm{RN},\nu|_{\mathcal{A}_\mathrm{RN}})=h_\varphi(\mathcal{A},\nu)$. Since $(\mathcal{A}_0)_\mathrm{RN}=\mathcal{A}_\mathrm{RN}$, for the same reason, $h_\varphi(\mathcal{A}_\mathrm{RN},\nu|_{\mathcal{A}_\mathrm{RN}})=h_\varphi(\mathcal{A}_0,\nu_0)$. Therefore, $h_\varphi(\mathcal{A}_0,\nu_0)=h_\varphi(\mathcal{A},\nu)$. 

    Assume that $h_\varphi(\mathcal{A}_0,\nu_0)=h_\varphi(\mathcal{A},\nu)$. First, let's prove that $\sigma_t|_M=\sigma_t^0|_M$. The following proof is inspired by the proof of \cite[Lemma 5.14]{DP20}.

    For any $t>0$, within $B(L^2(\mathcal{A}_0,\nu_0))$, we have 
    $$(\Delta_0+t)^{-1}=(e_0\Delta e_0+t)^{-1}\leq e_0(\Delta+t)^{-1}e_0,$$
    where the inequality can be deduced by the following formula of inverse of $2\times2$ block matrix:
$$\begin{pmatrix}
A&B\\
C&D
\end{pmatrix}^{-1}
=
\begin{pmatrix}
A^{-1}+A^{-1}BECA^{-1}&-A^{-1}BE\\
-ECA^{-1}&E
\end{pmatrix},
$$
    where $E=(D-CA^{-1}B)^{-1}$.

    Let $\{z_k\}\subset M$ be the sequence in the standard form of $\varphi$. Let $e$ be the orthogonal projection from $L^2(\mathcal{A},\nu)$ onto $L^2(M)$. By (\ref{int varphi}), we have
    $$
    \begin{aligned}
        0=&h_\varphi(\mathcal{A},\nu)-h_\varphi(\mathcal{A}_0,\nu_0)\\
        =&\int_{0}^{+\infty}\varphi(e[(\Delta+t)^{-1}-(\Delta_0+t)^{-1}]e)\d t \\
        =& \int_{0}^{+\infty}\sum_{k=1}^{\infty} \langle e[(\Delta+t)^{-1}-(\Delta_0+t)^{-1}]e \hat{z}^*_k, \hat{z}^*_k \rangle\d t\\
        =& \int_{0}^{+\infty}\sum_{k=1}^{\infty} \langle [e_0(\Delta+t)^{-1}e_0-(\Delta_0+t)^{-1}] \hat{z}^*_k, \hat{z}^*_k \rangle\d t\\
    \end{aligned}
    $$
    Since $ e_0(\Delta+t)^{-1}e_0\geq (\Delta_0 +t)^{-1}$ in $B(L^2(\mathcal{A}_0,\nu_0))$, we must have
    $$e_0(\Delta+t)^{-1}e_0  \hat{z}^*_k= (\Delta_0+t)^{-1} \hat{z}^*_k $$
    for any $t>0, k\geq1$. 
    
    By differentiating on $t$, we have
    $$e_0(\Delta+t)^{-2}e_0  \hat{z}^*_k= (\Delta_0+t)^{-2} \hat{z}^*_k.$$
    Hence
    $$
    \begin{aligned}
        \Vert (\Delta+t)^{-1}\hat{z}^*_k \Vert^2=&\langle e_0(\Delta+t)^{-2}e_0 \hat{z}^*_k, \hat{z}^*_k \rangle\\
        =&\langle (\Delta_0+t)^{-2} \hat{z}^*_k, \hat{z}^*_k \rangle \\
        = & \Vert (\Delta_0+t)^{-1}\hat{z}^*_k \Vert^2\\
        = & \Vert e_0(\Delta+t)^{-1}e_0  \hat{z}^*_k \Vert^2\\
        = & \Vert e_0(\Delta+t)^{-1} \hat{z}^*_k \Vert^2.\\
    \end{aligned}
    $$ 
    Therefore, for any $t>0, k\geq1$, 
    $$(\Delta+t)^{-1}\hat{z}^*_k= e_0(\Delta+t)^{-1} \hat{z}^*_k = e_0(\Delta+t)^{-1}e_0  \hat{z}^*_k=(\Delta_0+t)^{-1}\hat{z}^*_k.$$
    For any $n\in\mathbb{N}$, differentiating on $t$ for $n$ times, we have
    $$(\Delta+t)^{-n}\hat{z}^*_k=(\Delta_0+t)^{-n}\hat{z}^*_k.$$
    Therefore, for any polynomial $f\in\mathbb{C}[x]$, 
    \begin{equation}\label{f(D)}
     f((\Delta+1)^{-1})\hat{z}^*_k=f((\Delta_0+1)^{-1})\hat{z}^*_k.   
    \end{equation}

    Let $B([0,1])$ be the space of bounded Borel functions on $[0,1]$. Since $B([0,1])$ can be generated by $\mathbb{C}[x]$ as a function space closed under the uniformly bounded pointwise convergence along sequences (i.e. a sequence $f_n\to f$ if there exists a $C\geq 0$ such that $|f_n|\leq C$ for all $n$, and $(f_n)$ converges to $f$ pointwise), (\ref{f(D)}) also holds for any $f\in B([0,1])$. Take $f$ to be $f(x)=(\frac{1}{x}-1)^{it}$, we have
    $$\Delta^{it}\hat{z}^*_k=\Delta_0^{it}\hat{z}^*_k.$$
    That is 
    $$ \sigma_t(z_k^*)= \sigma^0_t(z_k^*)$$
    for any $k\geq 1, t\in \mathbb{R}$. 

    Since $\{z_k\}$ generates $M$, we have $\sigma_t|_M= \sigma_t^0|_M$ for any $t\in\mathbb{R}$. Therefore, according to Definition \ref{RN factor}, we have $(\mathcal{A}_0,\nu_0)_{\mathrm{RN}}=(\mathcal{A},\nu)_{\mathrm{RN}}$.

\end{proof}

Recall that in classical theory, a $(\Gamma,\mu)$-boundary is a $(\Gamma,\mu)$-space $(B_0,\nu_0)$ such that there exists a factor map from the $\mu$-Poisson boundary $(B,\nu)$ onto $(B_0,\nu_0)$. The following defines a generalization of $(\Gamma,\mu)$-boundary.

\begin{definition}
    Up to state preserving isomorphisms, a \textbf{$\varphi$-boundary} $(\CB_0,\zeta_0)$ is a von Neumann subalgebra of $(B_\varphi,\zeta)$ satisfying $(M,\tau)\subset(\CB_0,\zeta_0)\subset(\mathcal{B}_\varphi,\zeta)$. 
\end{definition}

The following theorem is the analogue of the fact that the Poisson boundary is the unique boundary with maximal entropy.
\begin{theorem} \label{bd max entropy}
   Let $\varphi\in \mathcal{S}_\tau(B(L^2(M,\tau))$ be a normal regular strongly generating hyperstate such that $H(\varphi)<+\infty$. Let $(\CB_\varphi,\zeta)$ be the Poisson boundary. Then for any $\varphi$-boundary $(\CB_0,\zeta_0) \subset(\mathcal{B}_\varphi,\zeta)$, $h_\varphi(\CB_0,\zeta_0)=h_\varphi(\CB_\varphi,\zeta)$ if and only if $(\CB_0,\zeta_0)=(\CB_\varphi,\zeta)$.
\end{theorem}

\begin{proof}
   Assume that $h_\varphi(\CB_0,\zeta_0)=h_\varphi(\CB_\varphi,\zeta)$. Then by Theorem \ref{RN=RN}, we have $(\CB_0,\zeta_0)_\mathrm{RN}=(\CB_\varphi,\zeta)_\mathrm{RN}$.

   Since $(\CB_\varphi)_\mathrm{RN}$ is invariant under the action of modular automorphism group, according to \cite[Theorem IX.4.2]{TakII}, there exists a normal conditional expectation $E:\CB_\varphi\to (\CB_\varphi)_\mathrm{RN}$. By the rigidity of Poisson boundary \cite[Theorem 4.1]{DP20}, we have $E=\id_{\CB_\varphi}$ and $(\CB_\varphi)_\mathrm{RN}=\CB_\varphi$. Hence
   $$(\CB_\varphi,\zeta)=(\CB_\varphi,\zeta)_\mathrm{RN}=(\CB_0,\zeta_0)_\mathrm{RN}\subset(\CB_0,\zeta_0).$$
   Therefore, $(\CB_0,\zeta_0)=(\CB_\varphi,\zeta)$.
\end{proof}

Take $(\CB_0,\zeta_0)$ to be $(M,\tau)$ in Theorem \ref{bd max entropy}, we have the following corollary, which is an analogue of the fact that the $(\Gamma,\mu)$-Poisson boundary is trivial if and only if $h(\mu)=0$. See also \cite[Corollary 5.16]{DP20} for the ``if'' direction.

\begin{corollary}
  Let $\varphi\in \mathcal{S}_\tau(B(L^2(M,\tau))$ be a normal regular strongly generating hyperstate such that $H(\varphi)<+\infty$. Then $\CB_\varphi=M$ if and only if $h(\varphi)=0$.   
\end{corollary}

\section{Trivial boundaries of amenable von Neumann algebras}

According to \cite{KV82}, a countable discrete group $\Gamma$ is amenable if and only if there exists an admissible measure $\mu \in \prob(\Gamma)$ (i.e. $\bigcup_{n\in\mathbb{N}}(\supp \mu)^n=\Gamma$) such that the Poisson boundary of $(\Gamma,\mu)$ is trivial. In this section, we will prove an analogue for tracial von Neumann algebras.

\begin{definition}
    For $\mu\in\prob(\mathcal{U}(M))$, we say that $\mu$ is an \textbf{atomic measure} if 
    $$\mu=\sum_{\mu(u)>0} \mu(u)\cdot \delta_{u}. $$
    In this case, we simply denote $\{u\in \mathcal{U}(M)\mid \mu(u)>0\}$ by $\suppess \mu$, which is a countable $\mu$-conull (hence dense) subset of $\supp \mu$.
\end{definition}

\begin{definition}
   For (atomic) $\mu \in \prob(\mathcal{U}(M))$, define the (normal) hyperstate $\varphi_\mu \in \mathcal{S}_\tau(B(L^2(M))$ to be
   $$\varphi_\mu(T)=\int_{\mathcal{U}(M)} \langle T u^*\hat{1},u^*\hat{1}\rangle \d \mu (u), \ T\in B(L^2(M)).$$ 
\end{definition}

According to \cite[Theorem 2.10]{DP20}, to find a normal hyperstate with trivial boundary, we only need to find a nice hyperstate $\varphi$, such that $\varphi^{*n}$ tends to be $M$-central when $n \to\infty$. Before constructing the $\varphi$ we need, we also need the following lemma to make sure such $\varphi$ can be obtained through a measure on $\mathcal{U}(M)$.

\begin{lemma}\label{3} Let $(M, \tau )$ be an amenable tracial von Neumann algebra with separable
predual and
$$\mathcal{S}_0=\{\varphi_\mu \in \mathcal{S}_\tau(B(L^2(M))\mid \mu \in \prob(\mathcal{U}(M)) \mathrm{\ is\  of\ finite\ support}\}.$$
Then there exists a hypertrace in the weak* closure of $\mathcal{S}_0$.
\end{lemma}
\begin{proof}
    By Connes' fundamental theorem (\cite{Co75}), $M$ is hyperfinite. Assume that $M=\bigvee_n Q_n$, where $Q_n$ is an increasing sequence of unital finite dimensional $\ast$-subalgebras of $M$.
    Then a conditional expectation $\Phi : B(L^2(M))\to M$ can be given by
    \begin{equation}\label{1}
        \Phi(T)=\lim_{n\to\omega}\int_{\mathcal{U}(Q_n)}(Ju^*J)T(JuJ)\d \mu_n(u) :=\lim_{n\to\omega} \Phi_n(T), \ T\in B(L^2(M)), 
    \end{equation}
    where $\mu_n $ is the unique Haar measure on compact group $\mathcal{U}(Q_n)$ and $\omega\in\beta \mathbb{N}\setminus \mathbb{N}$. 
    
    Obviously, each $\Phi_n$ is a contractive unital $M$-bimodular map, hence so is $\Phi$. For any $m\geq n$ and $u\in \mathcal{U}(Q_n)\subset \mathcal{U}(Q_m)$, by the right invariance of $\mu_m$, $\Phi_m(T)$ commutes with $JuJ$, hence so does $\Phi(T)$. Therefore  $\Phi(T)$  commutes with $JQ_nJ$, let $n \to +\infty$, we know that $\Phi(T)\in (JMJ)'=M$, so $\Phi$ is a conditional expectation from $B(L^2(M))$ to $M$.
    Then 
    $$\varphi_0(T)= \langle \Phi(T) \hat{1},\hat{1}\rangle=\tau(\Phi(T)), \ T\in B(L^2(M))$$
    defines an hypertrace. 
    
    By (\ref{1}), we also have
    $$\varphi_0(T)=\lim_{n\to\omega}\int_{\mathcal{U}(Q_n)}\langle T u^*\hat{1},u^*\hat{1}\rangle\d \mu_n(u)=\lim_{n\to\omega}\varphi_{\mu_n}(T).$$
    Therefore, $\varphi_0=\lim_{n\to\omega}\limits \varphi_{\mu_n}$ is in the weak* closure of $\{ \varphi_{\mu_n}\}$. 

    To show $\varphi_0$ is in the weak* closure of $\mathcal{S}_0$, we only need to show each $\varphi_{\mu_n}$ is in it. Fix a $n\in\mathbb{N}$. Since $\mathcal{U}(Q_n)$ is compact, for any $\epsilon > 0$, there exists $u_1,...,u_N\in \mathcal{U}(Q_n)$ such that $\mathcal{U}(Q_n)\subset \bigcup_{i=1}^{N}\limits B(u_i, \epsilon)$. Let
    $$E_1=B(u_1,\epsilon),\ E_i=B(u_i,\epsilon) \setminus \bigcup_{j<i}B(u_j,\epsilon)\ (i\geq 2).$$
    Define $\mu_n'\in \prob(\mathcal{U}(M))$ with finite support by
    $$\mu_n'=\sum_{i=1}^{N}\mu_n(E_i)\delta_{u_i}.$$
    Then for any $u\in E_i\subset B(u_i,\epsilon)$, $T \in B(L^2(M))$,
    $$|\langle T u^*\hat{1},u^*\hat{1}\rangle-\langle T u_i^*\hat{1},u_i^*\hat{1}\rangle|\leq |\langle T (u^*-u_i^*)\hat{1},u^*\hat{1}\rangle|+|\langle T u_i^*\hat{1},(u_i^*-u)\hat{1}\rangle|\leq 2\epsilon \|T\|.$$
    Therefore,
    $$|\varphi_{\mu_n}(T)-\varphi_{\mu_n'}(T)|\leq \sum_{i=1}^{N}\int_{E_i}|\langle T u^*\hat{1},u^*\hat{1}\rangle-\langle T u_i^*\hat{1},u_i^*\hat{1}\rangle|\d \mu_n(u)\leq 2\epsilon \|T\|.$$
    
    So for any $\epsilon>0$, there exist a $\varphi_{\mu_n'}\in \mathcal{S}_0$ such that $\|\varphi_{\mu_n}-\varphi_{\mu_n'}\|\leq 2 \epsilon$. Hence $\varphi_{\mu_n}$ is in the weak* closure of $\mathcal{S}_0$, and so is $\varphi_0$
\end{proof}

The following theorem is inspired by \cite[Theorem 4.3]{KV82}.

\begin{theorem}\label{4} 
    Let $(M, \tau )$ be an amenable tracial von Neumann algebra with separable predual. Then there exists an atomic measure $\mu\in\prob(\mathcal{U}(M))$, such that $\bigcup_{n\in\mathbb{N}}(\suppess\mu)^n$ is weakly dense in $\mathcal{U}(M)$ and the normal hyperstate $\varphi=\varphi_\mu$ satisfies $\varphi \geq \frac{1}{2}\langle \, \cdot \, \hat{1},  \hat{1} \rangle$, and for any $x \in M$, 
    $$\lim_{n \to \infty}\limits \|\varphi^{*n} (x \, \cdot \, )- \varphi^{*n} ( \, \cdot \, x) \|=0,$$
    which is equivalent to 
    $$\lim_{n \to \infty}\limits \|xA_n- A_n x \|_{1,Tr}=0,$$
    where  $A_n$ is the trace-class operator associated to $\varphi^{*n}$.
\end{theorem}

\begin{proof}
    By Lemma \ref{3}, there exists an hypertrace in the weak* closure of $\mathcal{S}_0$. Therefore, by the Hahn-Banach separation theorem, for any finite subset $E\subset M$ and $\epsilon > 0$, there exists a $\psi\in \mathcal{S}_0$ such that for any $x \in E$, 
    $$\|\psi (x \, \cdot \, )- \psi ( \, \cdot \, x) \|< \epsilon.$$
    
    We construct a sequence $\{\varphi_n\}\subset \mathcal{S}_0$ as follows:\\
    Let $\{1\}=K_1\subset K_2 \subset ...$ be an increasing sequence of finite subsets of $(M)_1$, such that $K:=\bigcup_{n=1}^{\infty}K_n$ is strongly dense in $(M)_1$. Let $\{1\}=U_1\subset U_2 \subset ...$ be an increasing sequence of finite subsets of $\mathcal{U}(M)$, such that $U:=\bigcup_{n=1}^{\infty}U_n$ is weakly dense in $\mathcal{U}(M)$. Take $\frac{1}{2}=t_1>t_2>...$ to be a decreasing series of positive real number, such that $\sum_{n=1}^{\infty}\limits t_n=1$. 
    
    For $n=1$, we take $\varphi_1$ associated to $\mu_1=\delta_{\{1\}}$ by
    $$\varphi_1(\, \cdot \, )= \langle \, \cdot \, \hat{1},  \hat{1} \rangle.$$
    
    Assume that we already have $\varphi_1,...,\varphi_{n-1}$ associated to $\mu_1,...,\mu_{n-1}$. We construct $\varphi_n$ as follows:\\
    Take $m_n\in \mathbb{N}$ such that
    $$(t_1+t_2+\cdots+t_{n-1})^{m_n}<2^{-n}.$$
    Let $F_n=(\{1\}\bigcup K_n \bigcup (\bigcup_{i=1}^{n-1}\limits \supp \mu_i))^{m_n}$ be a finite subset of $(M)_1$.

    We take $\varphi_{\mu'_n}\in \mathcal{S}_0$ satisfying for any $x \in F_n$, 
    $$\|\varphi_{\mu'_n} (x \, \cdot \, )- \varphi_{\mu'_n} ( \, \cdot \, x) \|< 2^{-n-1}.$$
    Take a finitely supported $\mu''_n\in \prob(\mathcal{U}(M))$ with $\supp \mu''_n=U_n$. Then take $$\mu_n=(1-2^{-n-2})\mu'_n+2^{-n-2}\mu''_n\in \prob(\mathcal{U}(M))$$ and $$\varphi_n =\varphi_{\mu_n}\in \mathcal{S}_0.$$ Then $U_n\subset \supp \mu_n$ and $\varphi_n$ satisfies that for any $x \in F_n$, 
    $$\|\varphi_{n} (x \, \cdot \, )- \varphi_{n} ( \, \cdot \, x) \|\leq (1-2^{-n-2})\|\varphi_{\mu'_n} (x \, \cdot \, )- \varphi_{\mu'_n} ( \, \cdot \, x) \|+2^{-n-2}\|\varphi_{\mu''_n} (x \, \cdot \, )- \varphi_{\mu''_n} ( \, \cdot \, x) \|< 2^{-n}.$$

    Now we have constructed $\{\varphi_n\}\subset\mathcal{S}_0$, satisfying that for any $n\geq 1$ and $x\in F_n$, $$\|\varphi_{n} (x \, \cdot \, )- \varphi_{n} ( \, \cdot \, x) \|<2^{-n}.$$
    
    Let $\varphi\in \mathcal{S}_\tau(B(L^2(M)))$ be
    $$\varphi=\sum_{i=1}^{\infty}t_i \varphi_i,$$ 
    which is associated to the atomic measure 
    $$\mu=\sum_{i=1}^{\infty}t_i \mu_i.$$
    We will prove $\mu$ and $\varphi$ are exactly what we want.
    
    First of all, $U=\bigcup_n U_n \subset \bigcup_n \supp \mu_n=\suppess \mu$. Since $U$ is weakly dense in $\prob(\mathcal{U}(M))$, we have that $\bigcup_{n\in\mathbb{N}}(\suppess\mu)^n$ is weakly dense in $\prob(\mathcal{U}(M))$.
    
    Secondly, we obviously have 
    $$\varphi \geq t_1\varphi_1 = \frac{1}{2}\langle \, \cdot \, \hat{1},  \hat{1} \rangle.$$
    
    Now let's prove that for $x\in K$,
    $$\lim_{n \to \infty}\limits \|\varphi^{*n} (x \, \cdot \, )- \varphi^{*n} ( \, \cdot \, x) \|=0.$$

    For a fixed $x \in K$, there exists a $n_0 \in \mathbb{N}$ such that $x\in K_n$ for any $n\geq n_0$. For any $n\geq n_0$, we consider $\|\varphi^{*m_n} (x \, \cdot \, )- \varphi^{*m_n} ( \, \cdot \, x) \|$. 

    For $k=(k_1,...,k_{m_n})\in {\mathbb{N}_+^{m_n}}$, denote $|k|:=\max_{j}\limits k_j$. Then
    $$
    \begin{aligned}
    \varphi^{*m_n}&=\sum_{k\in {\mathbb{N}_+^{m_n}}} t_{k_1}\cdots t_{k_{m_n}} \varphi_{k_1}*\cdots *\varphi_{k_{m_n}}\\
    &=\sum_{|k|\leq n-1 } t_{k_1}\cdots t_{k_{m_n}} \varphi_{k_1}*\cdots *\varphi_{k_{m_n}} + \sum_{|k|\geq n } t_{k_1}\cdots t_{k_{m_n}} \varphi_{k_1}*\cdots *\varphi_{k_{m_n}}\\
    :&=\psi_1+\psi_2.
    \end{aligned}
    $$

    Obviously, by how we pick $m_n$, 
    $$\|\psi_1\|\leq \sum_{|k|\leq n-1 } t_{k_1}\cdots t_{k_{m_n}}=(t_1+\cdots +t_{n-1})^{m_n} <2^{-n}.$$
    Therefore,
    $$\|\psi_1 (x \, \cdot \, )- \psi_1 ( \, \cdot \, x) \|\leq 2 \|\psi_1\| <2^{-n+1}.$$
    
    To prove $\|\psi_2 (x \, \cdot \, )- \psi_2 ( \, \cdot \, x) \| <C\cdot 2^{-n}$, it is enough to prove that for $|k|\geq n$,
    $$\|\varphi_{k_1}*\cdots *\varphi_{k_{m_n}} (x \, \cdot \, )- \varphi_{k_1}*\cdots *\varphi_{k_{m_n}}( \, \cdot \, x) \| <C\cdot 2^{-n}.$$
    For a fixed $k$ such that $|k|\geq n$, let $j$ be the lowest index such that the inequality $k_j\geq n$ holds. Then we can rewrite $\varphi_{k_1}*\cdots *\varphi_{k_{m_n}} $ in the form
    $$\varphi_{k_1}*\cdots *\varphi_{k_{m_n}}=\eta_1*\varphi_{k_j}*\eta_2.$$
    Since $\eta_1=\varphi_{k_1}*\cdots *\varphi_{k_{j-1}}\in \mathcal{S}_0$, we may assume that for $T\in B(L^2(M))$,
    $$\eta_1(T)=\sum_{i=1}^{N}\mu_{\eta_1}(u_i) \cdot \langle Tu^*_i \hat{1}, u^*_i \hat{1} \rangle, $$
    where $\{u_1,...,u_N\} =\supp \mu_{\eta_1} \subset \prod_{l=1}^{j-1}\limits \supp \mu_{k_l} \subset F_n$.
    For any $T\in B(L^2(M))$, 
    $$
    \begin{aligned}
    \eta_1*\varphi_{k_j}(T)&=\eta_1(\CP_{\varphi_{k_j}}(T))\\
    &=\sum_{i=1}^{N}\mu_{\eta_1}(u_i) \cdot \langle \CP_{\varphi_{k_j}}(T)u^*_i \hat{1}, u^*_i \hat{1} \rangle\\
    &=\sum_{i=1}^{N}\mu_{\eta_1}(u_i) \cdot \langle \CP_{\varphi_{k_j}}(u_iTu^*_i) \hat{1},  \hat{1} \rangle\\
    &=\sum_{i=1}^{N} \mu_{\eta_1}(u_i) \cdot \varphi_{k_j}(u_iTu^*_i).
    \end{aligned}
    $$
    Therefore, for any $T \in B(L^2(M))$ such that $\Vert T\Vert\leq 1$,
    $$
    \begin{aligned}
    |\eta_1*\varphi_{k_j}(xT)-\eta_1*\varphi_{k_j}(Tx)|=&|\sum_{i=1}^{N} \mu_{\eta_1}(u_i) \cdot \varphi_{k_j}(u_ixTu^*_i)-\sum_{i=1}^{N} \mu_{\eta_1}(u_i) \cdot \varphi_{k_j}(u_iTxu^*_i)|\\
    \leq  &\sum_{i=1}^{N} \mu_{\eta_1}(u_i) \cdot |\varphi_{k_j}(u_ixTu^*_i)-\varphi_{k_j}(xT)|\\
    &+ \sum_{i=1}^{N} \mu_{\eta_1}(u_i) \cdot |\varphi_{k_j}(u_iTxu^*_i)-\varphi_{k_j}(Tx)|\\
    &+|\varphi_{k_j}(xT)-\varphi_{k_j}(Tx)|.
    \end{aligned}
    $$
    Since $x\in K_n \subset F_{k_j}$, by how we pick $\varphi_n$, we have
    $$|\varphi_{k_j}(xT)-\varphi_{k_j}(Tx)|<2^{-k_j}\leq 2^{-n}.$$
    Since $u_i\in F_n \subset F_{k_j}$ and $\|xTu^*_i\|\leq 1$, we have
    $$|\varphi_{k_j}(u_ixTu^*_i)-\varphi_{k_j}(xT)|=|\varphi_{k_j}(u_i(xTu^*_i))-\varphi_{k_j}((xTu^*_i)u_i)|\leq \|\varphi_{k_j}(u_i\, \cdot \, )-\varphi_{k_j}(\, \cdot \, u_i)\|<2^{-k_j}\leq 2^{-n}.$$
    Similarly,
    $$|\varphi_{k_j}(u_iTxu^*_i)-\varphi_{k_j}(Txu^*_iu_i)| <   2^{-n}.$$
    Therefore, by the inequality above, we have
    $$|\eta_1*\varphi_{k_j}(xT)-\eta_1*\varphi_{k_j}(Tx)|<\sum_{i=1}^{N} \mu_{\eta_1}(u_i) \cdot  2^{-n}+\sum_{i=1}^{N} \mu_{\eta_1}(u_i) \cdot  2^{-n}+ 2^{-n}\leq 3 \cdot 2^{-n}.$$
    So
    $$\|\eta_1*\varphi_{k_j}(x\, \cdot \, )-\eta_1*\varphi_{k_j}(\, \cdot \, x)\|\leq 3 \cdot 2^{-n}.$$
    
    For any $T\in B(L^2(M))$,
    $$
    \begin{aligned}
    &\eta_1*\varphi_{k_j}*\eta_2(xT)-\eta_1*\varphi_{k_j}*\eta_2(Tx)\\
    =&\eta_1*\varphi_{k_j}(\CP_{\eta_2}(xT))-\eta_1*\varphi_{k_j}(\CP_{\eta_2}(Tx))\\
    =&\eta_1*\varphi_{k_j}(x\CP_{\eta_2}(T))-\eta_1*\varphi_{k_j}(\CP_{\eta_2}(T)x)\\
    =& ((\eta_1*\varphi_{k_j}(x\, \cdot \, )-\eta_1*\varphi_{k_j}(\, \cdot \, x))\circ \CP_{\eta_2})(T)
    \end{aligned}
    $$
    Therefore, 
    $$\|\eta_1*\varphi_{k_j}*\eta_2(x\, \cdot \, )-\eta_1*\varphi_{k_j}*\eta_2(\, \cdot \, x)\|\leq\|\eta_1*\varphi_{k_j}(x\, \cdot \, )-\eta_1*\varphi_{k_j}(\, \cdot \, x)\|\cdot\|\CP_{\eta_2}\| \leq 3\cdot 2^{-n}.$$

So now we already prove that for $k\in \mathbb{N}_+^{m_n}$ such that $|k|\geq n$,
$$\|\varphi_{k_1}*\cdots*\varphi_{k_{m_n}} (x \, \cdot \, )- \varphi_{k_1}*\cdots*\varphi_{k_{m_n}}( \, \cdot \, x) \| \leq 3 \cdot 2^{-n}.$$
Therefore, 
$$\|\psi_2 (x \, \cdot \, )- \psi_2 ( \, \cdot \, x) \| \leq 3\cdot 2^{-n}.$$
And we have
$$\|\varphi^{*m_n} (x \, \cdot \, )- \varphi^{*m_n} ( \, \cdot \, x) \|\leq \|\psi_1 (x \, \cdot \, )- \psi_1 ( \, \cdot \, x) \|+\|\psi_2 (x \, \cdot \, )- \psi_2 ( \, \cdot \, x) \|\leq 5\cdot 2^{-n}.$$
So 
$$\lim_{n \to \infty}\limits \|\varphi^{*m_n} (x \, \cdot \, )- \varphi^{*m_n} ( \, \cdot \, x) \|=0.$$
Since 
$$\varphi^{*(n+1)} (x \, \cdot \, )- \varphi^{*(n+1)} ( \, \cdot \, x)=\varphi^{*n} (\mathcal{P}_\varphi(x \, \cdot \, ))- \varphi^{*n} ( \mathcal{P}_\varphi(\, \cdot \, x))=\varphi^{*n} (x\mathcal{P}_\varphi( \, \cdot \, ))- \varphi^{*n} ( \mathcal{P}_\varphi(\, \cdot \, )x),$$
we have 
$$\|\varphi^{*(n+1)} (x \, \cdot \, )- \varphi^{*(n+1)} ( \, \cdot \, x)\|\leq \| \varphi^{*n} (x \, \cdot \, )- \varphi^{*n} ( \, \cdot \, x)\| \cdot \|\mathcal{P}_\varphi\|=\| \varphi^{*n} (x \, \cdot \, )- \varphi^{*n} ( \, \cdot \, x)\|.$$
Therefore, $\| \varphi^{*n} (x \, \cdot \, )- \varphi^{*n} ( \, \cdot \, x)\|$ is decreasing and there is a subsequence that converges to 0. We must have
$$\lim_{n \to \infty}\limits \|\varphi^{*n} (x \, \cdot \, )- \varphi^{*n} ( \, \cdot \, x) \|=0$$
for any $x \in K$.

Finally, we want to prove that for any $y\in (M)_1$, 
$$\lim_{n \to \infty}\limits \|\varphi^{*n} (y \, \cdot \, )- \varphi^{*n} ( \, \cdot \, y) \|=0.$$

Fixed a $y \in (M)_1$. Let $A_1$ be the associated trace-class operator of $\varphi$. For any $\epsilon>0$, since $K$ is strongly dense in $(M)_1$ (weakly density is not enough), there exists a $x \in K$ such that $\|(y-x)A_1\|_{1,Tr},\|A_1(y-x)\|_{1, Tr}<\epsilon$, which is equivalent to $\|\varphi((y-x)\, \cdot \, )\|,\|\varphi(\, \cdot \, (y-x))\|<\epsilon$. Then we have
$$\|\varphi^{*n} (y \, \cdot \, )- \varphi^{*n} ( \, \cdot \, y) \|\leq \|\varphi^{*n} (x \, \cdot \, )- \varphi^{*n} ( \, \cdot \, x) \|+ \|\varphi^{*n}((y-x)\, \cdot \, )\| + \|\varphi^{*n}(\, \cdot \, (y-x))\|.$$
Since 
$$\varphi^{*n}((y-x)\, \cdot \, )=\varphi(\mathcal{P}_\varphi^{n-1}((y-x)\, \cdot \, ))=\varphi((y-x)\mathcal{P}_\varphi^{n-1}(\, \cdot \, )),$$
we have
$$\|\varphi^{*n}((y-x)\, \cdot \, )\|\leq \|\varphi((y-x)\, \cdot \, )\|<\epsilon.$$
Similarly,
$$\|\varphi^{*n}(\, \cdot \, (y-x))\|<\epsilon.$$
Therefore, 
$$\limsup_{n\to+\infty}\limits \|\varphi^{*n} (y \, \cdot \, )- \varphi^{*n} ( \, \cdot \, y) \|\leq \lim_{n\to +\infty}\limits \|\varphi^{*n} (x \, \cdot \, )- \varphi^{*n} ( \, \cdot \, x) \|+2\epsilon =2\epsilon .$$
Let $\epsilon \to 0$, we have
$$\lim_{n\to+\infty}\limits \|\varphi^{*n} (y \, \cdot \, )- \varphi^{*n} ( \, \cdot \, y) \|=0.$$

So $\varphi$ satisfies all the conditions we want.

\end{proof}

Now we can prove the following main theorem of this section.

\begin{theorem}\label{amenable}
    Let $(M, \tau )$ be a tracial von Neumann algebra with separable predual. Then the following conditions are equivalent. 
    \begin{itemize}
        \item[(i)] There exists an atomic measure $\mu\in\prob(\mathcal{U}(M))$, such that $\bigcup_{n\in\mathbb{N}}(\suppess\mu)^n$ is weakly dense in $\prob(\mathcal{U}(M))$ and the normal hyperstate $\varphi=\varphi_\mu$ has trivial Poisson boundary;
        \item[(ii)] There exists a hyperstate $\varphi\in \mathcal{S}_\tau(B(L^2(M)))$ with trivial Poisson boundary;
        \item[(iii)] $(M, \tau )$ is amenable.
    \end{itemize}
\end{theorem}

\begin{proof}
    (i) $\Rightarrow$ (ii) is clear.

    (ii) $\Rightarrow$ (iii): Assume that the hyperstate $\varphi$ has trivial Poisson boundary. Then by the injectivity of Poisson boundaries  \cite[Proposition 2.4]{DP20}, $M=\har(\CP_\varphi)$ is injective (amenable).

    (iii) $\Rightarrow$ (i) is a direct corollary of \cite[Theorem 2.10]{DP20} and Theorem \ref{4}.
\end{proof}

\section{Choquet-Deny property of von Neumann algebras}

Recall that a locally compact group $G$ is called Choquet-Deny if for any admissible $\mu\in\prob(G)$, the Poisson boundary of $\mu$ is trivial. In this section, we will define the Choquet-Deny property for tracial von Neumann algebras and prove the equivalence between Choquet-Deny property and type $\mathrm{I}$, which is inspired by the fact that a finitely generated countable discrete group $\Gamma$ is Choquet-Deny if and only if $\Gamma$ is virtually nilpotent (see \cite{Ja04} and \cite{FHTV19}), and $\Gamma$ is virtually abelian if and only if $L(\Gamma)$ is of type $\mathrm{I}$ (see \cite{THOM64}).

\begin{definition}
    Let $(M, \tau )$ be a tracial von Neumann algebra with separable predual. We say that $(M,\tau)$ is \textbf{Choquet-Deny} if for any normal regular strongly generating hyperstate $\varphi\in \mathcal{S}_\tau(B(L^2(M)))$, one has $\CB_\varphi=M$.
\end{definition}

We only consider weakly separable $(M,\tau)$, because it is the only way that $M$ admits a normal strongly generating hyperstate.

According to \cite[Corollary 3.2]{DP20}, any abelian $(M,\tau)$ with separable predual is Choquet-Deny. Let's see more examples of Choquet-Deny and non-Choquet-Deny von Neumann algebras.

\begin{proposition}\label{Mn}
    For $n\in \mathbb{N}_+$, the type $\mathrm{I}_n$ factor $(M_n(\mathbb{C}), \frac{1}{n}\mathrm{Tr})$ is Choquet-Deny.
\end{proposition}

\begin{proof}
    Note that we represent $M_n(\mathbb{C})$ on $L^2(M_n(\mathbb{C}), \frac{1}{n}Tr)$. Fix a normal regular strongly generating hyperstate $\varphi$ on $B(L^2(M_n(\mathbb{C})))$. Then there exists a sequence $\{w_m\}\subset M_n(\mathbb{C})'=M_n(\mathbb{C})'\cap B(L^2(M_n(\mathbb{C})))$, such that $\sum_{m=1}^{\infty} w_m w_m^*=\sum_{m=1}^{\infty} w_m^* w_m=1$, $\{w_m\}$ generates $M_n(\mathbb{C})'$ as a weakly closed unital subalgebra, and
    $$\mathcal{P}_\varphi(T)=\sum_{m=1}^{\infty}w_m^*Tw_m,\ T\in B(L^2(M_n(\mathbb{C}))).$$

    Fix a $T\in \mathrm{Har}(\mathcal{P}_\varphi)$, let's prove that $T\in M_n(\mathbb{C})$. By considering $\mathrm{Re\ }T,\ \mathrm{Im\ }T \in \mathrm{Har}(\mathcal{P}_\varphi)$, we may assume that $T$ is self-adjoint. Since $L^2(M_n(\mathbb{C}))$ is of finite dimension, $T$ is of finite rank. Consider the spectral decomposition of $T$:
    $$T=\sum_{k=1}^{N}\lambda_k P_k,$$
    where $\lambda_1>\lambda_2>...>\lambda_N$ are eigenvalues of $T$ and for $1\leq k\leq N$, $P_k$ is the orthogonal projection corresponding to the $\lambda_k$-eigenspaces $V_k$. For any $v_1\in (V_1)_1$ the unit ball of $V_1$, we have
    $$\lambda_1=\langle Tv_1, v_1\rangle= \langle\sum_{m=1}^{\infty}\limits w_m^*Tw_mv_1, v_1\rangle\leq \lambda_1 \langle\sum_{m=1}^{\infty}\limits w_m^*w_mv_1, v_1\rangle = \lambda_1 \langle v_1, v_1\rangle = \lambda_1,$$
    where the inequality holds for $T\leq \lambda_1 \cdot 1$. The third inequality becomes an equality now, so for any $m\geq 1$, we have
    $$\langle w_m^*Tw_mv_1, v_1\rangle= \lambda_1 \langle w_m^*w_mv_1, v_1\rangle,$$
    which is equivalent to
    $$\langle Tw_mv_1, w_mv_1\rangle= \lambda_1 \langle w_mv_1, w_mv_1\rangle=\lambda_1 \|w_mv_1\|^2.$$
    Since $\lambda_1$ is the greatest eigenvalue of $T$, we must have $w_mv_1\in V_1$. Hence for any $m\geq 1$, $w_mV_1\subset V_1$. Since $\{w_m\}$ generates $M_n(\mathbb{C})'$ as a weakly closed unital subalgebra, we have $M_n(\mathbb{C})'V_1 \subset V_1$. Therefore, $P_1=P_{V_1}\in M_n(\mathbb{C})''=M_n(\mathbb{C}).$
    
    Assume that we already have $P_1,...,P_k\in M_n(\mathbb{C})\subset\mathrm{Har}(\mathcal{P}_\varphi)$, take 
    $$T_k=\sum_{j=k+1}^{N}\limits\lambda_j P_j=T-\sum_{j=1}^{k}\limits \lambda_j P_j \in \mathrm{Har}(\mathcal{P}_\varphi).$$ 
    Now $\lambda_{k+1}$ becomes the greatest eigenvalue of $T_k$. By the same proof, we have $P_{k+1}\in M_n(\mathbb{C})$. Hence by induction, for any $1\leq k \leq N$, $P_k\in M_n(\mathbb{C})$ and $T=\sum_{k=1}^{N}\limits\lambda_k P_k\in M_n(\mathbb{C})$.

    Therefore, $\mathrm{Har}(\mathcal{P}_\varphi)=M_n(\mathbb{C})$. So $M_n(\mathbb{C})$ is Choquet-Deny.
\end{proof}

\begin{proposition}\label{LS}
   The unique amenable $\mathrm{II}_1$ factor $L(S_\infty)$ is not Choquet-Deny.
\end{proposition}

\begin{proof}
    According to \cite{Ja04} and \cite{FHTV19}, a countable discrete group is Choquet-Deny if and only if it has no $\mathrm{ICC}$ quotients. But $S_\infty$ itself is $\mathrm{ICC}$, so $S_\infty$ is not Choquet-Deny. 

    Therefore, there exists an admissible $\mu \in \prob(S_\infty)$, i.e. $\supp \mu$ generates $S_\infty$ as semigroup, such that the $\mu$-Poisson boundary $(B,\nu_B)$ is not trivial. 

    Also view $\mu$ as an element in $\prob(\mathcal{U}(L(S_\infty)))$. Then $\varphi_\mu$ is a normal regular strongly generating hyperstate. According to \cite[Theorm 4.1]{Iz04}, the $\varphi_\mu$-Poisson boundary $\CB_{\varphi_\mu}$ is $L(S_\infty\curvearrowright B)$. 
    
    Since $(B,\nu_B)$ is not trivial, $L(S_\infty\curvearrowright B)\not=L(S_\infty)$. Hence $L(S_\infty)$ is not Choquet-Deny.
\end{proof}

Now we do some preparations for the proof of Choquet-Deny - type $\mathrm{I}$ equivalence.

\begin{proposition}\label{Z(M)}
    Let $(M,\tau)$ be a tracial von Neumann algebra, let $\varphi$ be a normal regular strongly generating hyperstate. Then $\mathrm{Har}(\mathcal{P}_\varphi)\subset Z(M)'$.   
\end{proposition}

\begin{proof}
    Fix a $u\in\mathcal{U}(Z(M))$, define $\Psi_u : \mathcal{B}_\varphi\to \mathcal{B}_\varphi$ to be
    $$\Psi_u(x)=uxu^*, \ x \in \mathcal{B}_\varphi.$$
    Then $\Psi_u$ is a normal u.c.p. map and $\Psi_u|_M= \mathrm{id}_M$. By the rigidity of Poisson boundary \cite[Theorem 4.1]{DP20}, we have $\Psi_u = \mathrm{id}_{\CB_\varphi}$.

    Therefore, for any $u\in\mathcal{U}(Z(M))$ and $x \in \mathcal{B}_\varphi$, we have $uxu^*=x$. By Poisson transform, which is $M$-bimodular, we know that for any $u\in\mathcal{U}(Z(M))$ and $T \in \mathrm{Har}(\mathcal{P}_\varphi)$, $uTu^*=T$. Therefore, $\mathrm{Har}(\mathcal{P}_\varphi)\subset Z(M)'$.
\end{proof}

For the next theorem, we refer to \cite[Chapter IV]{TakI} for details on direct integral theory.

\begin{theorem}\label{DI}
Let $(M,\tau)$ be a tracial von Neumann algebra with direct integral decomposition on Borel probability measure space $(X,\mu_X)$:
$$(M,\tau)=\int_X^\oplus (M_x,\tau_x)\d\mu_X(x).$$

Let $\varphi$ be a normal regular strongly
generating hyperstate on $B(L^2(M))$. By direct integral theory, $\varphi$ admits the following decomposition while restricting to the decomposable operators: 
$$\varphi|_{\int_X^\oplus B(L^2(M_x))\d\mu_X(x)}=\int_X^\oplus \varphi_x\d\mu_X(x),$$
where for $x\in X$, $\varphi_x$ is a normal $\tau_x$-hyperstate on $B(L^2(M_x))$.

Then we have 
$$\mathrm{Har}(\mathcal{P}_\varphi)=\int_X^\oplus \mathrm{Har}(\mathcal{P}_{\varphi_x})\d\mu_X(x)$$
and
$$\mathcal{B}_\varphi=\int_X^\oplus \mathcal{B}_{\varphi_x}\d\mu_X(x).$$
\end{theorem}

\begin{proof}
Obviously, the diagonal operator $L^\infty(X,\mu_X)$ is contained in $Z(M)$. By Proposition \ref{Z(M)},
$$\mathrm{Har}(\mathcal{P}_\varphi)\subset Z(M)' \subset L^\infty(X,\mu_X)'=\int_X^\oplus B(L^2(M_x))\d\mu_X(x).$$

Fix a $T=\int_X^\oplus T_x \in\int_X^\oplus B(L^2(M_x))\d\mu_X(x)$. Since for any $a,b\in M$,
$$    
\begin{aligned}
    \langle a\left(\int_X^\oplus \mathcal{P}_{\varphi_x}(T_x)\right)b \hat{1},\hat{1}\rangle 
    = &\int_X^\oplus \langle a(x)\mathcal{P}_{\varphi_x}(T_x)b(x) \hat{1}_{M_x},\hat{1}_{M_x} \rangle_{\tau_x} \d\mu_X(x) \\
    = &\int_X^\oplus \varphi_x(a(x)T_xb(x))  \d\mu_X(x) \\
    =&\varphi(aTb),
\end{aligned}
$$

we have
$$\mathcal{P}_\varphi(T)= \int_X^\oplus \mathcal{P}_{\varphi_x}(T_x).$$

Hence for $T=\int_X^\oplus T_x \in\int_X^\oplus B(L^2(M_x))\d\mu_X(x)$, $\mathcal{P}_\varphi(T)=T$ if and only if $\mathcal{P}_{\varphi_x}(T_x)=T_x$ for $\mu_X$-a.e. $x\in X$.

Therefore, $\mathrm{Har}(\mathcal{P}_\varphi)=\int_X^\oplus \mathrm{Har}(\mathcal{P}_{\varphi_x})\d\mu_X(x)$ and $\mathcal{B}_\varphi=\int_X^\oplus \mathcal{B}_{\varphi_x}\d\mu_X(x)$.
\end{proof}

Theorem \ref{DI} shows that the noncommutative Poisson boundary is somehow additive, indicating that in some cases, the study of noncommutative Poisson boundaries can be simplified by focusing the Poisson boundaries of tracial factors.

Now we are ready to prove the following main theorem.

\begin{theorem}\label{CD type I}
Let $(M, \tau )$ be a tracial von Neumann algebra with separable predual. Then the following conditions are equivalent. 
\begin{itemize}
    \item[(i)]$M$ is Choquet-Deny;
    \item[(ii)]For any atomic measure $\mu\in\prob(\mathcal{U}(M))$ such that $\suppess \mu$ generates $M$ as a weakly closed unital subalgebra, $\varphi_\mu$ has a trivial Poisson boundary;
    \item[(iii)]$M$ is of type $\mathrm{I}$.
\end{itemize}
\end{theorem}

\begin{proof}

We take the factor decomposition of $M$:
$$(M,\tau)=\int_X^\oplus (M_x,\tau_x)\d\mu_X(x),$$
where $L^\infty(X,\mu_X)=(Z(M),\tau|_{Z(M)})$ and for every $x\in M$, $(M_x,\tau_x)$ is a tracial factor with separable predual.\\

(i) $\Rightarrow$ (ii) is clear.\\

(ii) $\Rightarrow$ (iii): Assume that $M$ satisfies (ii). First, let's prove that for $\mu_X$-a.e. $x\in X$, $(M_x,\tau_x)$ is amenable. 

Take an atomic measure $\mu_0\in\prob(\mathcal{U}(M))$ such that $\suppess \mu_0$ generates $M$ as a weakly closed unital subalgebra. Then $\psi:=\varphi_{\mu_0}$ is a normal regular strongly generating hyperstate. Let 
$$\psi|_{\int_X^\oplus B(L^2(M_x))\d\mu_X(x)}=\int_X^\oplus \psi_x\d\mu_X(x)$$ 
be the decomposition of $\psi$ while restricting to the decomposable operators. 

Since $M$ satisfies (ii), we have $\mathrm{Har}(\mathcal{P}_\psi)=M$. By Theorem \ref{DI}, we also have
$$\int_X^\oplus \mathrm{Har}(\mathcal{P}_{\psi_x})\d\mu_X(x)=\mathrm{Har}(\mathcal{P}_\psi)=M=\int_X^\oplus M_x\d\mu_X(x).$$
Hence for $\mu_X$-a.e. $x\in X$, $\mathrm{Har}(\mathcal{P}_{\psi_x})=M_x$. By Theorem \ref{amenable}, for $\mu_X$-a.e. $x\in X$, $(M_x,\tau_x)$ is amenable.

Since an amenable tracial factor can only be $M_n(\mathbb{C})$ $(n\in\mathbb{N}_+)$ or $L(S_\infty)$, let
$$X_1=\{x\in X\mid \exists \ n\in\mathbb{N}_+,  M_x=M_n(\mathbb{C})\},$$
$$X_2=\{x\in X \mid M_x=L(S_\infty)\}.$$
Then $X_1\bigsqcup X_2$ is a conull subset of $X$. Now let's prove that $\mu_X(X_2)=0$.

Assume that $\mu_X(X_2)\not=0$. Let $\mu_{X_1}=\mu_X|_{X_1}/\mu_X(X_1)(=0,\ \mathrm{if} \ \mu_X(X_1)=0)$, $\mu_{X_2}=\mu_X|_{X_2}/\mu_X(X_2)$ and
$$(M_1,\tau_1)=\int_{X_1}^\oplus (M_x,\tau_x)\d\mu_{X_1}(x),$$
$$(M_2,\tau_2)=\int_{X_2}^\oplus (M_x,\tau_x)\d\mu_{X_2}(x)=L^\infty(X_2,\mu_{X_2})\Bar{\otimes}L(S_\infty).$$
Then 
$$(M,\tau)=(M_1\oplus M_2, (\mu_X(X_1)\cdot\tau_1)\oplus(\mu_X(X_2)\cdot\tau_2)).$$

Take an atomic measure $\mu_1\in\prob(\mathcal{U}(M_1))$ such that $\suppess \mu_1$ generates $M_1$ as a weakly closed unital subalgebra. Take an atomic measure $\mu_2\in\prob(\mathcal{U}(L^\infty(X_2)))$ such that $\suppess \mu_2$ generates $L^\infty(X_2)$ as a weakly closed unital subalgebra. 

By Proposition \ref{LS}, $L(S_\infty)$ is not Choquet-Deny. And according to the proof, there exists an atomic measure $\mu_3\in\prob(\mathcal{U}(L(S_\infty)))$ such that $\suppess \mu_3$ generates $L(S_\infty)$ as a weakly closed unital subalgebra, and $\varphi_{\mu_3}$ has non-trivial boundary. So there exists a $T_0 \in B(L^2(S_\infty)) \setminus L(S_\infty)$ such that
$$\mathcal{P}_{\varphi_{\mu_3}}(T_0)=T_0.$$

Now define $\mu \in \prob(\mathcal{U}(M))$ by
$$\mu = \frac{1}{2}\sum_{u\in \suppess\mu_1}\mu_1(u) \cdot \delta_{u \oplus 1_{M_2}}+\frac{1}{2}\sum_{\substack{v\in \suppess\mu_2\\ w \in \suppess \mu_3} }\mu_2(v)\mu_3(w)\cdot\delta_{1_{M_1} \oplus (v\otimes w)}.$$

We have that $\mu$ is atomic and
$$\suppess \mu=\{u \oplus 1_{M_2} \}_{u\in \suppess\mu_1}\cup \{1_{M_1} \oplus (v\otimes w)\}_{\substack{v\in \suppess\mu_2\\ w \in \suppess \mu_3}}.$$ 
By the generating property of $\mu_1,\mu_2,\mu_3$, we know that $\suppess \mu$ also generates $M$ as a weakly closed unital subalgebra .

Consider $0\oplus(\mathds{1}_{X_2}\otimes T_0)\in B(L^2(M))$, we have
$$
\begin{aligned}
&\mathcal{P}_{\varphi_\mu}(0\oplus(\mathds{1}_{X_2}\otimes T_0))\\
=&\frac{1}{2}(\mathcal{P}_{\varphi_{\mu_1}}\oplus (\mathrm{id}\otimes \mathrm{id}))(0\oplus(\mathds{1}_{X_2}\otimes T_0))\\
&+\frac{1}{2}(\mathrm{id}\oplus(\mathcal{P}_{\varphi_{\mu_2}}\otimes\mathcal{P}_{\varphi_{\mu_3}}))(0\oplus(\mathds{1}_{X_2}\otimes T_0))\\
=&0\oplus(\mathds{1}_{X_2}\otimes T_0)
\end{aligned}
$$
Hence $0\oplus(\mathds{1}_{X_2}\otimes T_0)\in \mathrm{Har}(\mathcal{P}_{\varphi_\mu})$. But $T_0\not \in L(S_\infty)$, so $0\oplus(\mathds{1}_{X_2}\otimes T_0)\not \in M$. Therefore, $\mathrm{Har}(\mathcal{P}_{\varphi_\mu}) \not =M$ and $M$ doesn't satisfies (ii), contradiction. 

Therefore, $\mu_X(X_2)=0$ and $M=M_1$ is of type $\mathrm{I}$.
\\

(iii) $\Rightarrow$ (i): Assume that $M$ is of type $\mathrm{I}$. Then for $\mu_X$-a.e. $x\in X$, $(M_x,\tau_x)$ is a finite type $\mathrm{I}$ factor. Hence after replacing $X$ with a conull subset, we may assume that for every $x\in X$, $(M_x,\tau_x)=(M_{n_x}(\mathbb{C}),\frac{1}{n_x}Tr)$ for some $n_x\in \mathbb{N}_+$.

Fix a normal regular strongly generating hyperstate $\varphi$, we want to prove that $\mathrm{Har}(\mathcal{P}_\varphi)=M$.

Let $$\varphi|_{\int_X^\oplus B(L^2(M_x))\d\mu_X(x)}=\int_X^\oplus \varphi_x\d\mu_X(x)$$ 
be the decomposition of $\varphi$ while restricting to the decomposable operators. We will prove that for $\mu_X$-a.e. $x\in X$, $\varphi_x$ is normal regular strongly generating.

Take $\{z_n\}\subset M$ such that $\sum_{n=1}^{\infty} z_n^*z_n=\sum_{n=1}^{\infty} z_nz_n^*=1$ and 
$$\mathcal{P}_\varphi(T)=\sum_{n=1}^{\infty}(Jz_n^*J)T(Jz_nJ),\ T\in B(L^2(M)).$$
Then for $\mu_X$-a.e. $x\in X$, 
\begin{equation}\label{2321}
    \sum_{n=1}^{\infty}\limits z_n^*(x)z_n(x)=\sum_{n=1}^{\infty}\limits z_n(x)z_n^*(x)=1_M(x)=1_{M_x}.
\end{equation}
 For any $T=\int_X^\oplus T_x \in\int_X^\oplus B(L^2(M_x))\d\mu_X(x)$, we have
$$\int_X^\oplus \varphi_x(T_x)\d\mu_X(x)=\varphi(T)=\sum_{n=1}^{\infty}\langle Tz_n^*\hat{1},z_n^*\hat{1}\rangle =\int_X^\oplus \sum_{n=1}^{\infty}\langle T_xz_n^*(x)\hat{1}_{M_x},z_n^*(x)\hat{1}_{M_x}\rangle_{\tau_x}\d\mu_X(x).$$
Therefore, for $\mu_X$-a.e. $x\in X$,
$$\varphi_x(\, \cdot \, )=\sum_{n=1}^{\infty}\langle \, \cdot \, z_n^*(x)\hat{1}_{M_x},z_n^*(x)\hat{1}_{M_x}\rangle_{\tau_x}.$$
And by (\ref{2321}), for $\mu_X$-a.e. $x\in X$, $\varphi_x$ is normal regular.

Now let's prove that for $\mu_X$-a.e. $x\in X$, $\varphi_x$ is strongly generating. Let $N\subset M$ be the weakly closed unital subalgebra generated by $\{z_n\}$. For $x\in X$, let $N_x\subset M_x$ be the weakly closed unital subalgebra generated by $\{z_n(x)\}$. Then
$$N\subset\int_X^\oplus N_x\d\mu_X(x)\subset \int_X^\oplus M_x\d\mu_X(x) =M.$$

Since $\varphi$ is strongly generating, we have $N=M$. Therefore, 
$$\int_X^\oplus N_x\d\mu_X(x)= \int_X^\oplus M_x\d\mu_X(x)$$
Hence for $\mu_X$-a.e. $x\in X$, $N_x=M_x$, inducing that $\varphi_x$ is strongly generating.

By Proposition \ref{Mn}, for every $x\in X$, $M_x=M_{n_x}(\mathbb{C})$ is Choquet-Deny. Also, for $\mu_X$-a.e. $x\in X$, $\varphi_x$ is normal regular strongly generating, so we have that for $\mu_X$-a.e. $x\in X$,
$$\mathrm{Har}(\mathcal{P}_{\varphi_x})=M_x.$$
Furthermore, by Theorem \ref{DI},
$$\mathrm{Har}(\mathcal{P}_\varphi)=\int_X^\oplus \mathrm{Har}(\mathcal{P}_{\varphi_x})\d\mu_X(x)=\int_X^\oplus M_x\d\mu_X(x)=M.$$

Therefore, $\mathrm{Har}(\mathcal{P}_\varphi)=M$ for any normal regular strongly generating $\varphi$. Hence $M$ is Choquet-Deny.
\end{proof}

\begin{remark}
As mentioned in the introduction part, for a countable discrete group $\Gamma$, the Choquet-Deny property of $L(\Gamma)$ is not equivalent to, but strictly stronger than the Choquet-Deny property of $\Gamma$. Actually, there exists no such property (CD) for tracial von Neumann algebras satisfying that $\Gamma$ is Choquet-Deny if and only if $L(\Gamma)$ has the property (CD). This is because a Choquet-Deny group and a non-Choquet-Deny group can have a same group von Neumann algebra. Consider two groups: the Heisenberg group $H_3(\mathbb{Z})$ and $\mathbb{Z}\times S_\infty$. Since $H_3(\mathbb{Z})$ is amenable and of type $\mathrm{II}_1$ \cite[Example 7.D.5 and Theorem 8.F.4]{BH20}, we have $L(H_3(\mathbb{Z}))=Z(L(H_3(\mathbb{Z})))\Bar{\otimes}\mathcal{R}$, where $\mathcal{R}$ is the unique AFD $\mathrm{II}_1$ factor. Since $Z(H_3(\mathbb{Z}))\cong \mathbb{Z}$, $Z(L(H_3(\mathbb{Z})))$ must be diffuse for having a diffuse subalgebra $L(\mathbb{Z})$. Therefore, $Z(L(H_3(\mathbb{Z})))\cong L^\infty([0,1])$ and $L(H_3(\mathbb{Z}))=L^\infty([0,1])\Bar{\otimes}\mathcal{R}$. We also have $L(\mathbb{Z}\times S_\infty)=L(\mathbb{Z})\Bar{\otimes} L(S_\infty)=L^\infty([0,1])\Bar{\otimes}\mathcal{R}$. Hence these two groups has the same group von Neumann algebra. However, following \cite{Ja04} and \cite{FHTV19}, since $H_3(\mathbb{Z})$ is finitely generated and nilpotent, it is Choquet-Deny; while $\mathbb{Z}\times S_\infty$ has the ICC quotient $S_\infty$, hence it is non-Choquet-Deny. Therefore, there exists no property for tracial von Neumann algebras that can perfectly match with the Choquet-Deny property of groups.
\end{remark}

\section*{Acknowledgements}
This paper was completed under the supervision of Professor Cyril Houdayer. Theorem \ref{amenable} and Theorem \ref{CD type I} were initially introduced by Professor Houdayer as conjectures. The author would like to express sincere gratitude to Professor Cyril Houdayer for numerous insightful discussions and valuable comments on this paper. The author would also like to thank Professor Amine Marrakchi for useful comments regarding this paper.

\bibliographystyle{alphaurl}
\bibliography{ref}

\end{document}